\documentclass[twoside,11pt]{article}
\usepackage{amsfonts,amscd,amssymb}
\newcommand{\seqn}{\begin{equation}}
\newcommand{\eeqn}{\end{equation}}
\newcommand{\seqna}{\begin{eqnarray}}
\newcommand{\eeqna}{\end{eqnarray}}

\newenvironment{proof}{{\it Proof.}}{\hfill $ \square $ \vskip 4mm }

\newtheorem{definition}{Definition}[section]
\newtheorem{lemma}[definition]{Lemma}
\newtheorem{proposition}[definition]{ Proposition}

\newtheorem{remark}[definition]{Remark}

\newtheorem{theorem}[definition]{ Theorem}

\let\BBox\Box
\def\Box{$\BBox$}

\def\hfl#1#2{\smash{\mathop{\hbox to 10mm{\rightarrowfill}}
\limits^{\scriptstyle#1}_{\scriptstyle#2}}}
\def\vfl#1#2{\llap{$\scriptstyle  #1$}\left\downarrow
\vbox to 9mm{}\right. \rlap{$\scriptstyle #2$}}
\def\diagram#1{\def\normalbaselines{\baselineskip=0pt\lineskip=13pt\lineskiplimit=1pt}
\matrix{#1}}
\def\lettrine#1#2#3{\noindent\hangindent#1\hangafter-#2
\hskip-#1\smash{\hbox to#1{#3\hfill}}\ignorespaces}
\def\limind{\mathop{\oalign{lim\cr
\hidewidth$\longrightarrow$\hidewidth\cr}}}

\def\vfl#1#2{\llap{$\scriptstyle #1$}\left\uparrow
\vbox to 6mm{}\right.\rlap{$\scriptstyle #2$}}

\def\kfl#1#2{\llap{$\scriptstyle #1$}\left\downarrow
\vbox to 8mm{}\right.\rlap{$\scriptstyle #2$}}

\title{Projectivity and flatness over the endomorphism ring of a finitely generated $(H,{\mathcal C})$-comodule}

\author{T. Gu\'ed\'enon}

\date{}

\begin{document}

\maketitle

\centerline{D\'epartement de Math\'ematiques}
\centerline{Universit\'e de Ziguinchor}
\centerline{S\'en\'egal}
\centerline{Email: thomas.guedenon@univ-zig.sn}

\date{}

\begin{abstract} Let $k$ be a commutative ring, $H$ a bialgebra over $k$, $A$ an $H$-algebra and $\mathcal C$ an $(H,A)$-coring, i.e a left $(H,A)$-bimodule which is an $A$-coring in a compatible way. Left $(H,A)$-bimodules and their deformations play a fundamental role in the study of the differential geometry of noncommutative manifolds. Let $\Lambda$ be a right $(H,{\mathcal C})$-comodule. The $k$-module $_HEnd^{\mathcal C}(\Lambda)$ of right $(H,{\mathcal C})$-colinear maps from $\Lambda$ to $\Lambda$ is a ring. Let us assume that $\mathcal C$ is flat as a left $A$-module. If $\Lambda$ is finitely generated (finitely presented) as a right $(H,{\mathcal C})$-comodule, we give necessary and sufficient conditions for projectivity and flatness of a module over $_HEnd^{\mathcal C}(\Lambda)$. If $\mathcal C$ contains a fixed $H$-grouplike element, we can replace $\Lambda$ with $A$.  
\end{abstract}


AMS Mathematical Subject Classification (2020) 16T05, 16T15, 16D40, 16W30.

\section{Introduction}
 
Let $k$ be a field of characteristic zero and $A$ a noncommutative Poisson algebra. We refer to Xu \cite{Xu}, Flato, Gerstenhaber, Voronov \cite{FGV}, Kubo \cite{Kubo1, Kubo2}, Yang,Yao, Ye \cite{YYY} and Bao, Yao, Ye \cite{BaoYe} for further information on noncommutative Poisson algebras and quasi-Poisson modules over noncommutative Poisson algebras. In \cite{Guedenon}, we have introduced the notion of a quasi-Poisson $A$-coring and of a right quasi-Poisson comodule over a quasi-Poisson $A$-coring. The universal enveloping algebra $U(A)$ of the Lie algebra $A$ is a cocommutative Hopf algebra, and a noncommutative Poisson algebra is a left $U(A)$-module algebra. This suggests to us that we can extend the results in \cite{Guedenon} to the case where $A$ is a left $H$-module algebra over a commutative ring $k$, for $H$ a Hopf algebra with a bijective antipode, or more generally for $H$ a bialgebra. Let $k$ be a commutative ring, $H$ a bialgebra and $A$ a left $H$-module algebra. A left {\it $(H,A)$-module} will be the usual left $A \#H$-module. A left-right {\it $(H,A)$-module} is a left $H$-module which is also a right $A$-module in a compatible way. A left $(H,A)$-bimodule is an $A$-bimodule which is a left $(H,A)$-module with the left $A$-structure and a left-right $(H,A)$-module with the right $A$-structure. A left $(H,A)$-bimodule is called a left $H$-module $A$-bimodule in \cite{AschSch, Schenkel 1, Schenkel 2}; and an $(H,A,A)$-module in \cite[Page 317]{Caen}, where $H$ is a Hopf algebra. Left $(H,A)$-bimodules and their deformations play a central role in the study of the differential geometry of noncommutative manifolds \cite{AschSch}. 

An $(H,A)$-coring is an $A$-coring which is a left $(H,A)$-bimodule such that the $A$-coring structure maps are homomorphisms of left $(H,A)$-bimodules (Definition 2.2). A right $(H,{\mathcal C})$-comodule over an $(H,A)$-coring ${\mathcal C}$ is a right ${\mathcal C}$-comodule and a left-right $(H,A)$-module such that the comodule structure map is a homomorphism of left-right $(H,A)$-modules (Definition 2.6). Let $\mathcal C$ be an $(H,A)$-coring. We denote by $_H{\mathcal M}^{\mathcal C}$ the category of right $(H,{\mathcal C})$-comodules and right $(H,{\mathcal C})$-colinear maps. For $M$ and $N$ two right $(H,{\mathcal C})$-comodules, we denote by $_HHom^{\mathcal C}(M,N)$ the $k$-module of right $(H,{\mathcal C})$-colinear maps from $M$ to $N$. The $k$-module $^{\prime}{\mathcal C}$ of left $A$-linear maps ${\mathcal C} \rightarrow A$ is an associative algebra called the left dual ring of ${\mathcal C}$. 

When $\mathcal C$ is an $(H,A)$-coring flat as a left $A$-module, we show that $_H{\mathcal M}^{\mathcal C}$ is a Grothendieck category (Lemma 2.7).

Let $H$ be a Hopf algebra with a bijective antipode. Then $^{\prime}{\mathcal C}$ is a left $(H,A)$-bimodule and a left $H$-module algebra (Lemma 2.9). We can consider the category $_H{\mathcal M}_{^{\prime}{\mathcal C}}$ of left-right $(H,{^{\prime}{\mathcal C}})$-modules which is isomorphic to the category $_{(^{\prime}{\mathcal C})^{op} \#H}{\mathcal M}$ of left modules over the smash product $(^{\prime}{\mathcal C})^{op} \#H$ \cite[Proposition 8.4.1]{Caen}. Furthermore, $_H{\mathcal M}^{\mathcal C}$ is a full Grothendieck subcategory of $_{(^{\prime}{\mathcal C})^{op} \#H}{\mathcal M}$. Let $\mathcal C$ be finitely generated and projective as a left $A$-module. Then we show that the categories $_H{\mathcal M}^{\mathcal C}$, $_H{\mathcal M}_{{^{\prime}{\mathcal C}}}$ and $_{(^{\prime}{\mathcal C})^{op} \#H}{\mathcal M}$ are isomorphic (Lemma 2.11). We get from \cite{CG} or \cite{Guedenon 1} necessary and sufficient conditions for a module over the endomorphism ring $_HEnd^{\mathcal C}(\Lambda)$ to be projective or flat, where $\Lambda$ is an object of $_H{\mathcal M}^{\mathcal C}$ that is finitely generated (resp. finitely presented) as a left $(^{\prime}{\mathcal C})^{op} \#H$-module. Using the techniques developed in \cite{GRio}, we will establish these results if $H$ is a bialgebra and $\mathcal C$ is flat as a left $A$-module without assuming that $\mathcal C$ is finitely generated. More precisely, let $H$ be a bialgebra, $ {\mathcal C}$ an $(H,A)$-coring which is flat as left $A$-module and $\Lambda$ a finitely generated (resp. finitely presented) object in $_H{\mathcal M}^{\mathcal C}$. We give necessary and sufficient conditions for a module over the endomorphism ring $_HEnd^{\mathcal C}(\Lambda)$ to be projective or flat (Theorem 3.2 and Theorem 3.3). Note that if $H$ is a Hopf algebra with a bijective antipode and if $\Lambda$ is an object of $_H{\mathcal M}^{\mathcal C}$ finitely generated (finitely presented) in $_{(^{\prime}{\mathcal C})^{op} \#H}{\mathcal M}$, then it is finitely generated (finitely presented) in $_H{\mathcal M}^{\mathcal C}$. So we get generalizations of our earlier works (\cite{CG} with S. Caenepeel), \cite{TG} and \cite{Guedenon} (see Subsections 5.1, 5.2 and 5.4). 

Let $H$ be a Hopf algebra with a bijective antipode. If $\mathcal C$ contains a fixed $H$-grouplike element $x$ (Definition 4.1), then $A$ is a right $(H,{\mathcal C})$-comodule that is cyclic as a left ${^{\prime}{\mathcal C}}^{op} \#H$-module (Lemma 4.5), and  the endomorphism ring $_HEnd^{\mathcal C}(A)$ is isomorphic to the subring $A^{H,co{\mathcal C}, x}$ of $(H,{\mathcal C}, x)$-coinvariants of $A$. We deduce from Theorems 3.2 and 3.3, necessary and sufficient conditions for a given module over $A^{H,co{\mathcal C}, x}$ to be projective or flat (Theorems 4.6 and 4.7). 

As the algebra $A^{co{\mathcal C},x}$ has played an important role for Galois corings, we hope that the algebra $A^{H,co{\mathcal C}, x}$ will do so for Galois corings with Hopf actions.

In many applications: $\mathcal C$ could be the trivial $(H,A)$-coring $A$ (Subsection 5.1),  we also consider the case where $A$ and $\Lambda$ are $H$-algebras and $A$ is an $(H,\Lambda)$-ring with a right $H$- grouplike character (Definition 5.2) note that (if $H$ is a Hopf algebra with a bijective antipode, the left dual $^{\prime}{\mathcal C}$ of an $(H,A)$-coring $\mathcal C$ with a fixed $H$-grouplike element is an $(H,A)$-ring with a right grouplike character); $H$ could be the base ring $k$ considered as a trivial bialgebra (Subsection 5.2); $H$ could be a Hopf algebra acting on a coalgebra $C$ (Subsection 5.3); $H$ could be the universal enveloping algebra $U(A)$ of a noncommutative Poisson algebra $A$ (Subsection 5.4), or the group algebra of a group $G$ (Subsection 5.5), or the universal enveloping algebra of a Lie algebra $\mathcal G$ (Subsection 5.6); $H$ could be a Hopf algebra acting trivially and coacting on an associative algebra $A$ in a compatible way (Subsection 5.7); $H$ could be the deformation (by a twist) of a Hopf algebra acting on an associative algebra (Subsection 5.9). We also consider the category $_{\mathcal G}{\mathcal M}_A^H$ of $(A,{\mathcal G},H)$-comodules, where $\mathcal G$ is an $H$-comodule Lie algebra, $A$ is an $H$-comodule algebra, and the $H$-action and the $\mathcal G$-action on $A$ are compatible (Subsection 5.8). 
 
Although the techniques are not new, the use of an $H$-action on a coring is certainly original and will be of interest for the theory of Hopf algebras and their actions on $A$-corings and on $A$-rings. Also there are a lot of interesting applications of our results. Unadorned $\otimes$  will be over $k$.

\section{Preliminary results} 

Let $k$ be a commutative ring and $A$ be a $k$-algebra. If $_A{\mathcal C}_A$ is an $(A , A)$-bimodule, then ${\mathcal C} \otimes_A{\mathcal C}$ is an $(A , A)$-bimodule in the natural way. Likewise, if $M$ is a right $A$-module, then $M \otimes_A{\mathcal C}$ is a right $A$-module. An $A$-coring $\mathcal{C}$ is an $A$-bimodule ${\mathcal C}$ together with two $A$-bimodule maps $\Delta_{\mathcal{C}} : \mathcal{C} \rightarrow \mathcal{C} \otimes_A \mathcal{C}$ and $\epsilon_{\mathcal{C}}: {\mathcal{C}}\rightarrow A$ such that   
$$(id_{\mathcal C} \otimes_A \Delta_{\mathcal C}) \circ \Delta_{\mathcal C}=(\Delta_{\mathcal C} \otimes_A id_{\mathcal C}) \circ \Delta_{\mathcal C} \quad \hbox{and}$$
$$(id_{\mathcal C} \otimes_A \epsilon_{\mathcal C}) \circ \Delta_{\mathcal C}=(\epsilon_{\mathcal C} \otimes_A id_{\mathcal C}) \circ \Delta_{\mathcal C}=id_{\mathcal C}.$$
The last two relations mean that the usual coassociativity and counit properties hold. Thus an $A$-coring is a coalgebra in the monoidal category $_A{\mathcal M}_A$ of $A$-bimodules. For more details on corings, we refer to \cite{br1}, \cite{br2}, \cite{br3}, \cite{brWi} and \cite{caeMilZhu}. Let ${\mathcal{C}}$ be an $A$-coring. A right $\mathcal{C}$-comodule is a right $A$-module $M$ together with a right $A$-linear map $\rho_{M , \mathcal{C}} : M\rightarrow M \otimes_A \mathcal{C}$ such that
$$(id_M \otimes_A \epsilon_{\mathcal{C}})\circ \rho_{M , \mathcal{C}}=id_M,
\quad \hbox{and} \quad (id_M \otimes_A \Delta_{\mathcal{C}})\circ \rho_{M ,
\mathcal{C}}=(\rho_{M , \mathcal{C}} \otimes_A id_{\mathcal{C}})\circ \rho_{M ,\mathcal{C}}.$$
We will use Sweedler-Heyneman notation but we will omit the symbol $\sum$:
$$\Delta_{\mathcal{C}}(c)=c_1 \otimes_A c_2 \quad \hbox{and} \quad \rho_{M , \mathcal{C}}(m)=m_{0} \otimes_A m_{1}.$$
A morphism of right $\mathcal{C}$-comodules or a right $\mathcal C$-colinear map $f : M \rightarrow N$ is a right $A$-linear map such that
$$\rho_{N , \mathcal{C}} \circ f=(f\otimes_A id_M)\circ \rho_{M , \mathcal{C}};$$
or equivalently, a right $A$-linear map such that
$$f(m)_0 \otimes_A f(m)_1=f(m_0) \otimes_A m_1.$$
We denote by $Hom^{\mathcal{C}}(M , N)$ the vector space of morphisms of right $\mathcal{C}$-comodules from $M$ to $N$, by $\mathcal{M}^{\mathcal{C}}$ the category formed by right $\mathcal{C}$-comodules and morphisms of right $\mathcal{C}$-comodules and by $\mathcal{M}$ the category of $k$-modules.

Let $A$ be a $k$-algebra. If ${\mathcal C}$ is an $A$-coring, we will denote by $\blacktriangleright$ its left $A$-action and by $\blacktriangleleft$ its right $A$-action. ${\mathcal C} \otimes_A {\mathcal C}$ means $({\mathcal C}, \blacktriangleleft) \otimes_A (\blacktriangleright,{\mathcal C})$. We write $^{\prime}{\mathcal C}={_A}Hom({\mathcal C} , A)$: $A$ and $\mathcal C$ are considered as left $A$-modules. So the $A$-module structure of $\mathcal C$ is $(a,c) \mapsto a \blacktriangleright c$. Then $^{\prime}{\mathcal C}$ is an associative ring with unit $\epsilon_{\mathcal C}$ (see \cite[Proposition 35]{caeMilZhu}): the multiplication is defined by
$$f\#g= g \circ(id_{\mathcal C} \otimes f) \circ \Delta_{\mathcal C},$$
or equivalently,
$$(f\#g)(c)=g(c_1 \blacktriangleleft (f(c_2))$$
for all left $A$-linear maps $f$, $g$: ${\mathcal C}\rightarrow A$ and $c \in {\mathcal C}$; where $\Delta_{\mathcal C}(c)=c_1 \otimes_A c_2$. There is an algebra homomorphism $i: A \rightarrow {^{\prime}}{\mathcal C}$ defined by $i(a)(c)=\epsilon_{\mathcal C}(c)a$ for $a \in A$ and $c \in {\mathcal C}$: for $a$, $a' \in A$ and $c \in {\mathcal C}$, we have 
$$i(a)(a'\blacktriangleright c)=\epsilon_{\mathcal C}(a' \blacktriangleright c)a=a'\epsilon_{\mathcal C}(c)a=a'[i(a)(c)].$$ So $i(a)$ is left $A$-linear, that is, $i(a) \in  {^{\prime}}{\mathcal C}$. We also have
$$\begin{array}{rcl}i(a)\#i(a')(c)&=&i(a')[c_1\blacktriangleleft (i(a)(c_2))]\\
&=&i(a')[c_1\blacktriangleleft ((\epsilon_{\mathcal C}(c_2)a)]\\
&=&\epsilon_{\mathcal C}[c_1\blacktriangleleft (\epsilon_{\mathcal C}(c_2)a)]a'\\
&=&\epsilon_{\mathcal C}(c_1)[\epsilon_{\mathcal C}(c_2)a]a'\\
&=&\epsilon_{\mathcal C}(c)aa'=i(aa')(c).\end{array}$$
Clearly, $i(1_A)=\epsilon_{\mathcal C}(c)$. Therefore $^{\prime}{\mathcal C}$ is an $A$-bimodule: the biaction is $[(af)a'](c)=f(c\blacktriangleleft a)a'$ for $a$, $a' \in A$ and $f \in ^{\prime}{\mathcal C}$.

Let $H$ be a bialgebra. An algebra $A$ is a left {\it $H$-module algebra} if $A$ is a left $H$-module satisfying  
$$h.(aa')=(h_1.a)(h_2.a') \quad \hbox{and} \quad h.1_A=\epsilon_H(h)1_A \quad \forall \quad a, a' \in A, h \in H.$$ In this case, we say that the $H$-action on $A$ is compatible with the  multiplication in $A$.
A homomorphism of $H$-module algebras is a homomorphism of $H$-modules which is also a homomorphism of $k$-algebras. 

Let $A$ be a left $H$-module algebra. Thus we can form the smash product algebra $A \#H$. The product in $A\#H$ is defined by 
$$(ah)(a'h')=(a(h'_1.a'))(h_2h') \quad \forall a,a' \in A,h,h' \in H.$$

From now an $H$-module means a left $H$-module, and an $H$-algebra means a left $H$-module algebra. Let $A$ be an $H$-algebra. A vector space $M$ is a left $A\#H$-module if and only if $M$ is a left $A$-module and an $H$-module such that 
$$h(am)=(h_1.a)(h_2m), \quad \forall a \in A, m \in M, h \in H \quad \eqno(1).$$ 

We say that $M$ is a left $(H,A)$-module if $M$ is a left $A\#H$-module. A homomorphism of left $(H,A)$-modules (or a left  $(H,A)$-linear map) from $M$ to $N$ is a homomorphism of left $A\#H$-modules from $M$ to $N$, that is, a homomorphism of left $A$-modules and of $H$-modules from $M$ to $N$. The category of left $A \#H$-modules will be denoted $_{A \#H}{\mathcal M}$: its morphismes are the homomorphisms of  left $(H,A)$-modules.

A $k$-module $M$ is a left-right $(H,A)$-module if $M$ is a right $A$-module and an $H$-module satisfying the following compatibility condition:
$$h(ma)=(h_1m)(h_2.a) \quad \forall \quad a \in A, h \in H \quad m \in M \eqno(2).$$ 
A homomorphism of left-right $(H,A)$-modules (or a left-right $(H,A)$-linear map) from $M$ to $N$ is a homomorphism of right $A$-modules and of $H$-modules from $M$ to $N$, i.e.,
$$f(ma)=f(m)a, \quad \hbox{and} \quad f(hm)=hf(m), \quad m \in M, h \in H.$$
The category of left-right $(H,A)$-modules will be denoted $_H{\mathcal M}_A$: its morphisms are the homomorphisms of left-right $(H,A)$-modules. By \cite[Proposition 8.4.1]{Caen}, the category $_H{\mathcal M}_A$ is isomorphic to the catgory $_{A^{op} \#H}{\mathcal M}$. The product in $A^{op} \#H$ is defined by
$$(ah)(a'h')=([h_2.a')a](h_1h') \quad a,a' \in A^{op}, h,h' \in H.$$ 

We say that a $k$-module $M$ is a left $(H,A)$-bimodule if $M$ is an $A$-bimodule in such a way that with the right $A$-action, it is a left-right $(H,A)$-module and with the left $A$-action, it is a left $(H,A)$-module. In other words, a $k$-module $M$ is a left $(H,A)$-bimodule if $M$ is an $A$-bimodule and an $H$-module such that the equations $(1)$ and $(2)$ are satisfied. Given two left $(H,A)$-bimodules $M$ and $N$, a homomorphism $f$ of left $(H,A)$-bimodules from $M$ to $N$ is a homomorphism $f$ of $A$-bimodules which is also a homomorphism of $H$-modules from $M$ to $N$, i.e., $f(ma)=f(m)a$, $f(am)=af(m)$ and $f(hm)=hf(m)$ for all $h \in H$, $a \in A$ and $m \in M$. So a homomorphism $f$ of left $(H,A)$-bimodules from $M$ to $N$ is a homomorphism of left-right $(H,A)$-modules from $M$ to $N$ which is also a homomorphism of left $A$-modules, or equivalently, a homomorphism of left $(H,A)$-modules from $M$ to $N$ which is also a homomorphism of right $A$-modules, or equivalently, a $k$-linear map from $M$ to $N$ which is also a right $A$-linear map, a left $A$-linear map and an $H$-linear map. We denote by ${}_{H,A}{\mathcal M}_A $ the category of left $(H,A)$-bimodules with morphisms the homomorphisms of left $(H,A)$-bimodules. The $H$-algebra $A$ itself is a left $(H,A)$-bimodule. If $A$ is an $H$-algebra, then $A \otimes_kA$ is a left $(H,A)$-bimodule: the actions of $A$ are given by $$a''(a \otimes a')=(a''a) \otimes a', (a \otimes a')a''=a \otimes (a'a''), \quad \forall \quad a,a',a'' \in A,$$ and the $H$-action is the diagonal action :
$$h.(a \otimes a')=(h_1.a) \otimes (h_2.a'), \quad \forall \quad a,a' \in A, h \in H.$$

Let $H$ be a Hopf algebra with a bijective antipode. Then every left $H$-module algebra $A$ has a natural structure of a right $H^{cop}$-module algebra, where $H^{cop}$ is the Hopf algebra with the same multiplication as $H$ but with the co-opposite comultiplication: the right $H^{cop}$-action is given by 
$$a\bullet h=S_H^{-1}(h).a, \quad \forall \quad a \in A, h \in H^{cop},$$ and we can form the smash product algebra $H^{cop} \#A$. Recall that the product in $H^{cop}\#A $ is defined by 
$$(ha)(h'a')=(hh'_1)[(a \bullet h'_2)a']=(hh'_1)[(S_H^{-1}(h'_2).a)a' \quad \forall \quad a,a' \in A, \quad h,h' \in H^{cop}.$$
Every left-right $(H,A)$-module $M$ is a right $H^{cop} \#A$-module: the right $H$-action is given by
$$mh=S_H^{-1}(h)m, \quad \forall m \in M, h \in H^{cop}.$$ 

Every left $(H,A)$-bimodule $M$ is a left $A \#H$-module and a right $H^{cop} \#A$-module with the right $H$-action given as above.

From now unless otherwise stated, $H$ is a bialgebra and $A$ is an $H$-module algebra.

\begin{lemma} (1) Let $M$ be a left-right $(H,A)$-module and $N$ a left $(H,A)$-bimodule. Then $M \otimes_A N$ is a left-right $(H,A)$-module: the right $A$-module structure is the natural one and the $H$-action is given by the diagonal action: 
$$h(m \otimes_A n)=(h_1m) \otimes_A(h_2n).$$

(2) Let $M$ be a left $(H,A)$-bimodule and $N$ a left-left $(H,A)$-module. Then $M \otimes_A N$ is a left-left $(H,A)$-module: the left $A$-module structure is the natural one and the $H$-action is given by the diagonal action.

(3) Let $M$ and $N$ be left $(H,A)$-bimodules. Then $M \otimes_A N$ is a left $(H,A)$-bimodule: the $A$-biaction is the natural one and the $H$-action is the diagonal action.
\end{lemma}

\begin{proof} The results are well known in the theory of Hopf algebra action on rings.
\end{proof}

From Lemma 2.1, we deduce that $({}_{H,A}{\mathcal M}_A, \otimes_A,A)$ is a monoidal category. 

\begin{definition} An $(H,A)$-coring is an $A$-coring $({\mathcal C}, \blacktriangleright, \blacktriangleleft)$ which is also an $H$-module in such a way that it is  a left $(H,A)$-bimodule and the  homomorphisms of $A$-bimodules $\Delta_{\mathcal{C}} : \mathcal{C} \rightarrow \mathcal{C} \otimes_A \mathcal{C}$ and
$\epsilon_{\mathcal{C}}: {\mathcal{C}}\rightarrow A$ are also homomorphisms of $H$-modules.
In other words,  $\Delta_{\mathcal{C}} $ and $\epsilon_{\mathcal{C}}$ are homomorphisms of left $(H,A)$-bimodules. 
\end{definition}

We can call also an $(H,A)$-coring an $H$-module $A$-coring. Note that 
an $(H,A)$-coring is a coalgebra in the monoidal category ${}_{H,A}{\mathcal M}_A$. 

Let $\mathcal C$ be an $(H,A)$-coring. From the above definition, we have
$$\Delta_{\mathcal C}(h.c)=(h_1.c_1) \otimes (h_2.c_2), \quad \epsilon_{\mathcal C}(h.c)=h.\epsilon_{\mathcal C}(c) \quad \forall h \in H, c \in {\mathcal C}.$$
It follows from Lemma 2.1 that if $M$ is any left-right $(H,A)$-module, then $M \otimes_A{\mathcal C}$ is a left-right $(H,A)$-module with the $H$-diagonal action.

Let us denote by $\overline{A}$ the opposite algebra of the associative algebra $A$. If $a$ is an element of $A$, the notation $\overline{a}$ is the element $a$ considered as an element of $\overline{A}$. We know that the $k$-linear map $a \mapsto  \overline{a}$ is an antihomomorphism of algebras $A$ to  $\overline{A}$. The following three lemmas are examples of $(H,A)$-corings. We will give other examples in section 5.

\begin{lemma} (1) $\overline{A}$ is a left $(H,A)$-bimodule: $a. \overline{a'}.a''= \overline{aa'a''}$, $h.\overline{a}= \overline{h.a}$ for all $ h \in H$, $a',a'' \in A$ and $\overline{a} \in \overline{A}$.

(2) Let $S$ be a coalgebra over $k$. Then ${\mathcal C}=A \otimes_k S \otimes_k\overline{A}$ is an $(H,A)$-coring: the $A$-biaction is 
$$a' \blacktriangleright (a \otimes s \otimes \overline{b}) \blacktriangleleft a''=(a'a) \otimes s \otimes \overline{ba''}), \forall a,a',a'',b \in A, s \in S,$$ the $H$-module structure is the diagonal action:    
$$ h(a \otimes s\otimes \overline{b})=(h_1.a) \otimes s\otimes \overline{h_2.b}, \forall a,b \in A, s \in S, h \in H,$$  
$$\Delta_{\mathcal{C}}(a \otimes s\otimes \overline{b})=a \otimes s_1 \otimes 1_A \otimes s_2 \otimes \overline{b}, \forall a,b \in A, s \in S,$$ we have identified ${\mathcal C} \otimes_A{\mathcal C}$ with $A \otimes_k S\otimes_k \overline{A} \otimes_k S\otimes_k \overline{A}$ 
and $$\epsilon_{\mathcal{C}}(a \otimes s\otimes \overline{b})=\epsilon_S(s)ab, \forall s \in S, a, b \in A.$$   

(3) ${\mathcal C}=A \otimes_k\overline{A}$ is an $(H,A)$-coring.

\end{lemma}

\begin{proof} (1) The proof is easy and well known in the theory of Hopf algebra actions on rings.

(2) Clearly, $A \otimes_k S\otimes_k\overline{A}$ is an $A$-bimodule for the given structures. By (1), $\overline{A}$ is an $H$-module. It is well known that $A \otimes_k S\otimes_k\overline{A}$ is an $H$-module for the given $H$-action. Let $h \in H$, $a,a',b \in A$. We have 
$$\begin{array}{rcl} h.[a' \blacktriangleright(a \otimes s\otimes \overline{b})] &=& h.(a'a \otimes s\otimes \overline{b})\\
&=& h_1.(a'a) \otimes s\otimes \overline{h_2.b} \\
&=& (h_1.a')(h_2.a) \otimes s\otimes \overline{h_3.b} \\
&=& (h_1.a')[(h_2.a) \otimes s\otimes \overline{h_3.b}] \\
&=& (h_1.a') \blacktriangleright [h_2.a(\otimes s\otimes \overline {b}] \end{array}$$
and
$$\begin{array}{rcl} h.[(a \otimes s\otimes \overline{b}) \blacktriangleleft a'] &=& h.(a \otimes s\otimes \overline{ba'}) \\
&=& (h_1.a) \otimes s\otimes \overline{h_2.(ba')} \\
&=& (h_1.a) \otimes s\otimes \overline{(h_2.b)(h_3.a')} \\
&=& (h_1.a) \otimes s\otimes \overline{(h_2.b)(h_3.a')} \\
&=& [(h_1.a) \otimes s\otimes \overline{(h_2.b)}]\blacktriangleleft (h_3.a') \\
&=& [h_1.(a \otimes s\otimes \overline{b}]\blacktriangleleft (h_2.a').\end{array}$$
Thus $A \otimes_k S\otimes_k\overline{A}$ is a left $(H,A)$-bimodule. We have
$$\begin{array}{rcl}\epsilon_{\mathcal{C}}[h.(a \otimes s\otimes \overline{b})]&=& \epsilon_{\mathcal{C}}((h_1.a) \otimes s \otimes \overline{h_2.b})\\
&=& \epsilon_S(s)(h_1.a)(h_2. b) \\
&=& \epsilon_S(s)(h.(ab)) =h.(\epsilon_S(s)ab)=h.\epsilon_{\mathcal{C}}(a \otimes s\otimes \overline{b}), \end{array}$$
So $\epsilon_{\mathcal C}$ is a homomorphism of $H$-modules. Now we have
$$\epsilon_{\mathcal{C}}[(a \otimes s \otimes \overline{b}) \blacktriangleleft a']= \epsilon_{\mathcal{C}}(a \otimes s\otimes \overline{ba'})=\epsilon_S(s)aba'=\epsilon_{\mathcal{C}}(a \otimes s\otimes \overline{b})a';$$
so $\epsilon_{\mathcal{C}}$ is right $A$-linear. $\epsilon_{\mathcal{C}}$ is clearly left $A$-linear.

We also have
$$\begin{array}{rcl} \Delta_{\mathcal{C}}[h.(a \otimes s \otimes \overline{b})]&=& \Delta_{\mathcal{C}}[(h_1.a) \otimes s \otimes \overline{h_2.b}] \\
&=&[(h_1.a) \otimes s_1 \otimes_A \otimes s_2 \otimes \overline{h_2.b}] \\
&=& h.(a \otimes s_{1}\otimes 1_A \otimes s_{2}\otimes \overline{b} )\\
&=& h\Delta_{\mathcal{C}}(a \otimes s\otimes \overline{b});\end{array}$$
so $\Delta_{\mathcal{C}}$ is a homomorphism of $H$-modules. We have
$$\begin{array}{rcl} \Delta_{\mathcal{C}}[(a \otimes s \otimes \overline{b}) \blacktriangleleft a']&=& \Delta_{\mathcal{C}}(a \otimes s \otimes \overline{ba'}) \\
&=& (a \otimes s_{1}\otimes 1_A \otimes s_{2}\otimes \overline{ba'}) \\
&=& (a \otimes s_{1}\otimes 1_A \otimes s_{2}\otimes \overline{b})a'\\
&=&\Delta_{\mathcal{C}}(a \otimes s \otimes \overline{b})a';
\end{array}$$
so $\Delta_{\mathcal{C}}$ is right $A$-linear. Clearly, $\Delta_{\mathcal{C}}$ is left $A$-linear.

(3) Take $S=k$ in (2).
\end{proof}

\begin{lemma} (1) Let $S$ be a coalgebra over $k$. Then ${\mathcal C}=A \otimes_k S \otimes_k A$ is an $(H, A)$-coring: the $A$-bimodule action  is given by 
$$a' \blacktriangleright (a \otimes s \otimes b)\blacktriangleleft a''=(a'a) \otimes s \otimes (ba''), \forall a,a',a'',b \in A, s \in S,$$ the $H$-module structure is the diagonal action : 
$$h.(a \otimes s\otimes b)=(h_1.a) \otimes s \otimes (h_2.b), \forall h \in H, a,b \in A, s \in S,$$  
$$\Delta_{\mathcal{C}}(a \otimes s\otimes b)=a \otimes s_1 \otimes 1_A \otimes s_2  \otimes b, \forall a,b \in A, s \in S,$$ we have identified ${\mathcal C} \otimes_A{\mathcal C}$ with $A \otimes_k S\otimes_k A \otimes_k S \otimes_k A$,
and 
$$\epsilon_{\mathcal{C}}(a \otimes s\otimes b)=a\epsilon_S(s)b, \forall a, b \in A, s \in S.$$ 

(2) ${\mathcal C}=A \otimes_k A$ is an $(H,A)$-coring. 

\end{lemma}

\begin{proof} The proof is similar to that of Lemma 2.3.
\end{proof}

Let $M$ be a left-right $(H,A)$-module. The $H$-invariant elements of $M$ is
$$M^H=\{m \in M; hm=\epsilon_H(h)m \quad \forall \quad h \in H \}.$$
It is clear that $A^H$ is a trivial $H$-subalgebra of $A$. We will give our last example of an $(H,A)$-coring.

\begin{lemma} Let $B$ be an $H$-subalgebra of $A^H$  and ${\mathcal C}=A \otimes_BA$. Then $\mathcal C$ is an $(H,A)$-coring: the $A$-biaction is the natural one, the $H$-module structure is the diagonal action, $\Delta_{\mathcal{C}}(a \otimes_B a')=a \otimes_B 1_A \otimes_B a'$ and $\epsilon_{\mathcal{C}}(a \otimes_B a')=aa'$ for all $a,a' \in A$: we have identified ${\mathcal C} \otimes_A{\mathcal C}$ with $A \otimes_BA\otimes_BA$. 
\end{lemma}

\begin{proof} It is well known that  ${\mathcal C}=A \otimes_BA$ is an $A$-coring for the given structures. For $a,a' \in A,b \in B$ and $h \in H$, we have
$$\begin{array}{rcl} h.[(ab) \otimes_B a']&=& [h_1.(ab)] \otimes_B (h_2.a')\\
&=& [(h_1.a)(h_2.b)] \otimes_B (h_3.a')\\
&=& (h_1.a) \otimes_B[(h_2.b)(h_3.a')]\\
&=& (h_1.a) \otimes_B[(h_2.(ba')]\\
&=& h.(a \otimes_B(ba')).\end{array}$$ 
So the $H$-action is well defined. It is clear that $A \otimes_BA$ is an $H$-module for the diagonal action. As in Lemma 2.3, we can show that $\Delta_{\mathcal{C}}$ and $\epsilon_{\mathcal{C}}$ are $H$-module maps. We have
$$\begin{array}{rcl} h.[(a \otimes_B a') \blacktriangleleft a'']&=& h.(a \otimes_B a'a'')\\
&=& (h_1.a) \otimes_B h_2.(a'a'')\\
&=& (h_1.a) \otimes_B [h_2.a')(h_3.a'')]\\
&=& [(h_1.a) \otimes_B [h_2.a')] \blacktriangleleft(h_3.a'')\\
&=& [h_1.(a \otimes_B a')] \blacktriangleleft(h_2.a'').\end{array}$$
Similarly, we show that 
$$ h.[a'' \blacktriangleright (a \otimes_B a')]=(h_1.a'') \blacktriangleright [h_2.( a \otimes_B a')].$$
\end{proof}

\begin{definition} Let $\mathcal C$ be an $(H,A)$-coring. A right $(H,{\mathcal C})$-comodule $M$ is a right $\mathcal C$-comodule $M$ which is also a left-right $(H,A)$-module and the right $A$-linear map $\rho_{M , {\mathcal C}}$ is also a homomorphism of left-right $(H,A)$-modules. 
\end{definition}

From Definition 2.6, $M$ is a right $(H,{\mathcal C})$-comodule if
$$(id_M\otimes_A \Delta_{\mathcal C}) \circ \rho_{M , {\mathcal C}}=(\rho_{M , {\mathcal C}} \otimes_A id_{\mathcal C}) \circ \rho_{M , {\mathcal C}}, (id_M \otimes_A \epsilon_{\mathcal C})\circ \rho_{M , {\mathcal C}}=id_M$$ and 
$$\rho_{M , {\mathcal C}}(ma)=m_0 \otimes_A (m_1a), \quad \hbox{and} \quad \rho_{M , {\mathcal C}}(hm)=(h_1m_0) \otimes_A (h_2m_1), \quad \forall m \in M,a \in A;$$ where $\rho_{M , {\mathcal C}} : M \rightarrow M \otimes_A {\mathcal C}; m \mapsto m_0 \otimes_A m_1$.

Note that a right $(H,{\mathcal C})$-comodule is a right $\mathcal C$-comodule in the monoidal category $_H{\mathcal M}$ of left $H$-modules. So we can also call a right $(H,{\mathcal C})$-comodule an $H$-module right $\mathcal C$-comodule.

Let $\mathcal C$ be an $(H,A)$-coring. Then $\mathcal C$ is a right $(H,{\mathcal C})$-comodule with $\rho_{{\mathcal C}, {\mathcal C}}=\Delta_{\mathcal C}$. We will see in section 4 a condition which ensures that the $H$-algebra $A$ is a right $(H,{\mathcal C})$-comodule. Let $M$ be a right $(H,{\mathcal C})$-comodule. By the proof of \cite[18.9(1)]{brWi}, $M \otimes_A {\mathcal C}$ is a right ${\mathcal C}$-comodule with the ${\mathcal C}$-coaction given by $id_M \otimes \Delta_{\mathcal C}$. Since $id_M \otimes \Delta_{\mathcal C}$ is an $H$-linear map, $M \otimes_A {\mathcal C}$ is a right $(H,{\mathcal C})$-comodule.

Let $M$ and $N$ be right $(H,{\mathcal C})$-comodules. A homomorphism of right $(H,{\mathcal C})$-comodules (or a right $(H,{\mathcal C})$-colinear map) $f: M \rightarrow N$ is a left-right $(H,A)$-linear map such that
$$\rho_{N , {\mathcal C}} \circ f=(f\otimes_A id_M)\circ \rho_{M , {\mathcal C}}.$$
So a right $(H,{\mathcal C})$-colinear map is a right $\mathcal C$-colinear map which is also a homomorphism of $H$-modules. We can define in a similar way a left $(H,{\mathcal C})$-comodule for an $(H,A)$-coring $\mathcal C$ and a homomorphism of  left $(H,{\mathcal C})$-comodules. We denote the set of homomorphisms of right $(H,{\mathcal C})$-comodules from $M$ to $N$ by $_HHom^{\mathcal C}(M , N)$. So $_HHom^{\mathcal C}(M , N)$ is a $k$-submodule of the $k$-module $Hom^{\mathcal C}(M , N)$ of right $\mathcal C$-colinear maps from $M$ to $N$. Let us denote by $_H{\mathcal M}^{\mathcal C}$ the category formed by right $(H,{\mathcal C})$-comodules and homomorphisms of right $(H,{\mathcal C})$-comodules. Then $_H{\mathcal M}^{\mathcal C}$ is a subcategory of the category ${\mathcal M}^{\mathcal C}$ of right ${\mathcal C}$-comodules.

Let $\mathcal C$ be an $(H,A)$-coring. So $\mathcal C$ is an $A$-coring. Let us consider the left dual $^{\prime}{\mathcal C}$ of ${\mathcal C}$. We recall that $^{\prime}{\mathcal C}$ is an algebra under the product defined by
$$(f\#g)(c)=g(c_1 \blacktriangleleft  (f(c_2))$$
for all left $A$-linear maps $f,g: {\mathcal C}\rightarrow A$ and $c \in {\mathcal C}$.

\begin{lemma} Let $\mathcal C$ be an $(H,A)$-coring. Assume that $\mathcal C$ is left $A$-flat. Then $_H{\mathcal M}^{\mathcal C}$ is a Grothendieck category.
\end{lemma}

\begin{proof} It is well known that $_H{\mathcal M}$ is a Grothendieck category. By \cite[18.6 or 18.13]{brWi}, the category ${\mathcal M}^{\mathcal C}$ is a Grothendieck category since $\mathcal C$ is left $A$-flat. 

(1) A 0-object exists in $_H{\mathcal M}^{\mathcal C}$. If $M$ and $N$ are right $(H,{\mathcal C})$-comodules, then $_HHom^{\mathcal C}(M,N)$ is an additive abelian group. The composition of morphisms in $_H{\mathcal M}^{\mathcal C}$ is $k$-bilinear. If $M$ and $N$ are two objects of $_H{\mathcal M}^{\mathcal C}$, the direct sum $M \oplus N$ is an object of $_H{\mathcal M}^{\mathcal C}$. It follows that $_H{\mathcal M}^{\mathcal C}$ is an additive category. 

(2) Let $f: M \rightarrow N$ in $_H{\mathcal M}^{\mathcal C}$. Then $Ker f$ and $coker f$ are right $(H,{\mathcal C})$-comodules. We have $M/Ker f \simeq Im f$ (the first isomorphism theorem) in $_H{\mathcal M}^{\mathcal C}$. It follows that $_H{\mathcal M}^{\mathcal C}$ is an abelian category.

(3) If $\{M_{\lambda}, \lambda \in \Lambda\}$ is a family of right $(H,{\mathcal C})$-comodules, then $\oplus M_{\lambda}$ is a right $(H,{\mathcal C})$-comodule. It follows that $_H{\mathcal M}^{\mathcal C}$ is a cocomplete abelian category (an AB-3 category). 

(4) Direct limit is exact in $_H{\mathcal M}^{\mathcal C}$. So $_H{\mathcal M}^{\mathcal C}$ is an (AB-V)-category. 

All the properties from (1) to (4) are true because they are true in $_H{\mathcal M}$ and ${\mathcal M}^{\mathcal C}$.

(5) If $M$ is an object of $_H{\mathcal M}^{\mathcal C}$, then so is $M \otimes {\mathcal C}$ with $\mathcal C$-coaction induced by $\mathcal C$. As in \cite[Proposition 5]{caeMilZhu}, we can show that $\mathcal C$ is a generator of $M \otimes {\mathcal C}$ and that the $\mathcal C$-subcomodules of $\mathcal C$ form a family of generators for $_H{\mathcal M}^{\mathcal C}$.
\end{proof}

In a Grothendieck category $\mathcal A$, an object $\Lambda$ is finitely generated if for each directed family $\{\Lambda_i\}_{i \in I}$ of subobjects of $\Lambda$, with $\Lambda=\cup_{i \in I}\Lambda_i$, there exists $j \in I$ such that $\Lambda= \Lambda_j$. An object $\Lambda$ is finitely presented if $\Lambda$ is finitely generated and every epimorphism $ N \rightarrow \Lambda$ with $N$ finitely generated has a finitely generated kernel \cite{CuadraSimson}. It is known that $\Lambda$ is finitely presented (finitely generated) if and only if the functor $Hom_{\mathcal A}(\Lambda,-)$ commutes with inductive limits (direct sums) \cite{Gabriel}, \cite[Chapter 2]{Popescu} and \cite[Chapter 5]{Stenstrom}. 

\begin{lemma} Let $\mathcal C$ be an $(H,A)$-coring. Assume that $\mathcal C$ is left $A$-flat. Let $P$ and $I$ be objects of $_H{\mathcal M}^{\mathcal C}$. Then

(1) the functor $_HHom^{\mathcal C}(P , -): {}_H{\mathcal M}^{\mathcal C} \rightarrow {\mathcal M}$ is covariant and left exact. It is exact if and only if $P$ is projective in $_H{\mathcal M}^{\mathcal C}$.

(2) the functor $_HHom^{\mathcal C}(- , I): {}_H{\mathcal M}^{\mathcal C} \rightarrow {\mathcal M}$ is contravariant and left exact. It is exact if and only if $I$ is injective in $_H{\mathcal M}^{\mathcal C}$.
\end{lemma}

\begin{lemma} Let $H$ be a Hopf algebra with a bijective antipode.  

(1) $^{\prime}{\mathcal C}$ is a left $(H,A)$-bimodule: the $A$-biaction is defined by 
$$(af)(c)=f(c \blacktriangleleft a), (fa)(c)=f(c)a, \forall a\in A, c \in {\mathcal C}, f \in {}^{\prime}{\mathcal C}$$ 
and the $H$-action is given by : 
$$(hf)(c)=(h_2f(S_H^{-1}(h_1).c), \forall h \in H, c \in {\mathcal C}, f \in {}^{\prime}{\mathcal C}.$$

(2) $^{\prime}{\mathcal C}$ is an $H$-algebra, and we can form the smash product $^{\prime}{\mathcal C} \#H$.

(3) The algebra homomorphism $i: A \rightarrow {^{\prime}{\mathcal C}}$ defined by $i(a)(c)=\epsilon_{\mathcal C}(c)a$ is a homomorphism of $H$-modules, that is, $i$ is an $H$-algebra homomorphism. We have $af=i(a) \#f$ and $fa=f \# i(a)$.

\end{lemma}

\begin{proof} (1) It is well known that $hf$ is left $A$-linear, that is, $hf \in {^{\prime}{\mathcal C}}$. It is also known that $^{\prime}{\mathcal C}$ is an $H$-module. Let $a,a' \in A$, $c \in {\mathcal C}$ and $f \in {^{\prime}{\mathcal C}}$. We have
$$\begin{array}{rcl}[h(af)](c)&=& h_2[(af)(S_H^{-1}(h_1).c)]\\
&=& h_2[f((S_H^{-1}(h_1).c) \blacktriangleleft a))]\\
&=& h_3[f((S_H^{-1}(h_2).c) \blacktriangleleft (\epsilon_H(h_1)a))]\\
&=& h_4[f((S_H^{-1}(h_3).c) \blacktriangleleft (S_H^{-1}(h_2)h_1.a))]\\
&=& h_3[f(S_H^{-1}(h_2).(c \blacktriangleleft (h_1.a)))]\\
&=&  (h_2f)(c \blacktriangleleft (h_1.a))\\
&=& [(h_1.a)(h_2f)](c),\end{array}$$
so the equation $(1)$ is satisfied. We also have
$$\begin{array}{rcl} [h(fa')](c)&=& h_2[(fa')(S_H^{-1}(h_1).c)]\\
&=& h_2[f((S_H^{-1}(h_1).c)a')]\\
&=& h_2f(S_H^{-1}(h_1).c)(h_3.a')\\
&=& [(h_1f)(c)](h_2.a')\\
&=& [(h_1f)(h_2.a')](c),\end{array}$$
so the equation $(2)$ is satisfied.

(2) Let $h \in H$, $c \in {\mathcal C}$ and $f,g \in {^{\prime}{\mathcal C}}$. We have
$$\begin{array}{rcl}[(h_1f) \#(h_2g)](c) &=&(h_2g)[c_1 \blacktriangleleft((h_1f)(c_2))]\\
&=&h_4(g[S_H^{-1}(h_3).(c_1 \blacktriangleleft(h_2(f(S_H^{-1}(h_1).c_2)))])\\
&=&h_5(g[(S_H^{-1}(h_4).c_1) \blacktriangleleft S_H^{-1}(h_3)(h_2(f(S_H^{-1}(h_1).c_2)))])\\
&=&h_3(g[(S_H^{-1}(h_2).c_1) \blacktriangleleft (f(S_H^{-1}(h_1).c_2))])\\
&=&h_2(g[(S_H^{-1}(h_{12}).c_1) \blacktriangleleft (f(S_H^{-1}(h_{11}).c_2))])\\
&=&h_2(g[(S_H^{-1}(h_{1})_1.c_1) \blacktriangleleft (f(S_H^{-1}(h_{1})_2.c_2))])\\
&=&h_2(g[(S_H^{-1}(h_{1}).c)_1 \blacktriangleleft (f((S_H^{-1}(h_{1}).c)_2))])\\
&=&h_2[(f\#g)(S_H^{-1}(h_{1}).c)]\\
&=&[h(f\#g)](c). \end{array}$$
It follows that $(h_1f) \#(h_2g)=h(f\#g)$.

(3) Let $h \in H$, $c \in {\mathcal C}$ and $a \in A$. We have $$\begin{array}{rcl}h(i(a))(c)&=& h_2[i(a)(S_H^{-1}(h_1).c)]\\
&=& h_2[\epsilon_{\mathcal C}(S_H^{-1}(h_1).c)a]\\
&=& h_2[\epsilon_{\mathcal C}(S_H^{-1}(h_1).c)](h_1.a)\\
&=&[\epsilon_{\mathcal C}(h_2S_H^{-1}(h_1).c)](h_1.a)\\
&=&[\epsilon_{\mathcal C}(\epsilon_H(h_2)c)](h_1.a)\\
&=&\epsilon_{\mathcal C}(c)(h.a)\\
&=& i(h.a)(c).\end{array}$$
So $i$ is a homomorphism of $H$-modules.
\end{proof}

In Lemma 2.9, if we consider the right dual ${\mathcal C}^{\prime}$ of $\mathcal C$ instead of the left dual, then we don't need to assume that the antipode of $H$ is bijective. In this case, the $H$-action on ${\mathcal C}^{\prime}$ is given by 
$$(hf)(c)=h_1f(S_H(h_2).c); \forall h \in H, c \in {\mathcal C}, f \in {\mathcal C}^{\prime}.$$
The $A$-biaction on ${\mathcal C}^{\prime}$ is defined by 
$$(fa)(c)=f(a \blacktriangleright c), (af)(c)=af(c), \forall a\in A, c \in {\mathcal C}, f \in {}^{\prime}{\mathcal C}$$ 

\begin{lemma} Let $H$ be a Hopf algebra with a bijective antipode. Let $\mathcal C$ be an $(H,A)$-coring. Then every right $(H,{\mathcal C})$-comodule $M$ is an object of $_H{\mathcal M}_{^{\prime}{\mathcal C}}$: the $^{\prime}{\mathcal C}$-action is defined by $mf=m_0 f(m_1)$ for all $f \in {^{\prime}}{\mathcal C}$ and $m\in M$ (so $mi(a)=ma$).
\end{lemma}

\begin{proof} By \cite[Proposition 38]{caeMilZhu} or \cite[19.1]{brWi}, $M$ is a right $^{\prime}{\mathcal C}$-module. By the assumptions, $M$ is an $H$-module. For $h \in H$, $f \in {^{\prime}{\mathcal C}}$ and $m \in M$, we have
$$\begin{array}{rcl} h(mf)&=& h(m_0 f(m_1))\\
&=&  (h_1m_0)(h_2(f(m_1)))\\
&=&  (h_{1}m_0)[h_3f(\epsilon_H(h_{2})m_1)]\\
&=&  (h_{1}m_0)[h_4f(S_H^{-1}(h_3)h_{2}m_1)]\\
&=&  (h_{1}m_0)[h_4f(S_H^{-1}(h_3).(h_{2}m_1))]\\
&=&  (h_1m)_0[h_3f(S_H^{-1}(h_2).(h_1m)_1)]\\
&=&  (h_1m)_0[(h_2f)((h_1m)_1)]\\
&=& (h_1m)(h_2f). \end{array}$$
\end{proof}

Let $H$ be a Hopf algebra with a bijective antipode, and let $\mathcal C$ be an $(H,A)$-coring which is left $A$-projective. We can show as in \cite[19.2]{brWi} that 
$_H{\mathcal M}^{\mathcal C}$ is a full subcategory of $_H{\mathcal M}_{^{\prime}{\mathcal C}}$, i.e., $$_HHom^{\mathcal C}(M , N)={}_HHom_{^{\prime}{\mathcal C}}(M , N) \quad \hbox{for any} \quad M, N \in {}_H{\mathcal M}^{\mathcal C}.$$ We recall that $$_HHom_{^{\prime}{\mathcal C}}(M , N)={}_{(^{\prime}{\mathcal C})^{op} \#H}Hom(M,N)$$
since $_H{\mathcal M}_{^{\prime}{\mathcal C}}$ and $_{(^{\prime} {\mathcal C})^{op}\# H}{\mathcal M}$ are isomorphic categories \cite[Proposition 8.4.1]{Caen}. As consequences, 

(1) an object of $_H{\mathcal M}^{\mathcal C}$ that is projective in $_H{\mathcal M}_{^{\prime}{\mathcal C}}$ is projective in $_H{\mathcal M}^{\mathcal C}$. 

(2) an object of $_H{\mathcal M}^{\mathcal C}$ that is finitely generated in $_{(^{\prime} {\mathcal C})^{op}\# H}{\mathcal M}$ is finitely generated in $_H{\mathcal M}^{\mathcal C}$: indeed, a direct sum of objects in $_H{\mathcal M}^{\mathcal C}$ is a direct sum of objects in $_{(^{\prime} {\mathcal C})^{op}\# H}{\mathcal M}$ and the functor $_{(^{\prime}{\mathcal C})^{op} \#H}Hom( \Lambda,-)$ commutes with direct sums if $\Lambda$ is finitely generated in $_{(^{\prime} {\mathcal C})^{op}\# H}{\mathcal M}$. We deduce that the functor $_HHom^{\mathcal C}( \Lambda , -)$ commutes with direct sums in $_H{\mathcal M}^{\mathcal C}$. So $\Lambda$ is finitely generated in $_H{\mathcal M}^{\mathcal C}$.    

(3) an object of $_H{\mathcal M}^{\mathcal C}$ that is finitely presented in $_{(^{\prime} {\mathcal C})^{op}\# H}{\mathcal M}$ is finitely presented in $_H{\mathcal M}^{\mathcal C}$: indeed, a direct limit of objects in $_H{\mathcal M}^{\mathcal C}$ is a direct limit of objects in $_{(^{\prime} {\mathcal C})^{op}\# H}{\mathcal M}$ and the functor $_{(^{\prime}{\mathcal C})^{op} \#H}Hom( \Lambda,-)$ commutes with direct limits if $\Lambda$ is finitely presented in $_{(^{\prime} {\mathcal C})^{op}\# H}{\mathcal M}$. We deduce that the functor $_HHom^{\mathcal C}( \Lambda , -)$ commutes with direct limits in $_H{\mathcal M}^{\mathcal C}$. So $\Lambda$ is finitely presented in $_H{\mathcal M}^{\mathcal C}$.   

\begin{lemma} Let $H$ be a Hopf algebra with a bijective antipode. Let $\mathcal C$ be an $(H,A)$-coring. We have a functor $F:{}_H{\mathcal M}^{\mathcal C} \rightarrow {}_H{\mathcal M}_{^{\prime}{\mathcal C}}$. This functor is an isomorphism if $\mathcal C$ is finitely generated and projective as a left $A$-module.    
\end{lemma}

\begin{proof} Let $M$ be an object of $_H{\mathcal M}^{\mathcal C}$. By Lemma 2.10, $M$ is an object of $_H{\mathcal M}_{^{\prime}{\mathcal C}}$ with the same $H$-action and with the right-$^{\prime}{\mathcal C}$-action defined by $mf=m_0 f(m_1)$ for all $f \in {^{\prime}}{\mathcal C}$ and $m\in M$. Setting $F(M)=M$, we get the desired functor. Now let $\mathcal C$ be finitely generated and projective as a left $A$-module, and let $\{c_k, f_k, k=1,2, ..., m\}$ be a finite dual basis for $\mathcal C$, $c_k \in {\mathcal C}$, $f_k \in {}^{\prime}{\mathcal C}$. 
By \cite[Proposition 38]{caeMilZhu}, we have $c=\sum_{k}f_k(c) \blacktriangleright c_k$ and $f=\sum_{k}f_k.[f(c_k)]$. We have
$$h.c=\sum_{k}h.(f_k(c) \blacktriangleright  c_k)=\sum_{k}f_k(h.c) \blacktriangleright c_k, \quad \hbox{and}$$
$$\sum_{k}[h_1.f_k(c) \blacktriangleright (h_2.c_k)]=\sum_{k}f_k(h.c) \blacktriangleright c_k.$$
For $M$ an object of $_H{\mathcal M}_{{^{\prime}{\mathcal C}}}$, let $G(M)=M$, with $\mathcal C$-comodule structure $\rho_{M, \mathcal C}(m)=\sum_k mf_k \otimes_Ac_k$ and the same $H$-action. By \cite[Proposition 38]{caeMilZhu}, $M$ is a right $\mathcal C$-comodule. We have
$$ \begin{array}{rcl} \rho_{M, \mathcal C}(hm)&=& \sum_k(hm)f_k \otimes c_k\\
&=& \sum_k(h_1m)\epsilon_H(h_2)f_k \otimes c_k\\
&=& \sum_k(h_1m)(h_2S_H(h_3)f_k) \otimes c_k\\
&=& \sum_kh_1[m(S_H(h_2)f_k)] \otimes c_k\\
&=& \sum_{k,l}h_1[m\big(f_l.((S_H(h_2)f_k)(c_l))\big)] \otimes c_k\\
&=& \sum_{k,l}(h_1m)\big(h_{2}[f_l.((S_H(h_3)f_k)(c_l))]\big) \otimes c_k\\
&=& \sum_{k,l}(h_1m)[(h_{2}f_l).[h_3((S_H(h_4)f_k(c_l))]] \otimes c_k\\
&=& \sum_{k,l}(h_1m)(h_{2}f_l) \otimes [h_3((S_H(h_4)f_k)(c_l))] \blacktriangleright c_k\\
&=& \sum_{k,l}(h_1m)(h_{2}f_l) \otimes [h_3(S_H(h_4)\{f_k(S_H^{-1}(S_H(h_5)).c_l)\})] \blacktriangleright c_k\\
&=& \sum_{k,l}(h_1m)(h_{2}f_l) \otimes f_k(h_3.c_l) \blacktriangleright c_k\\
&=& \sum_{k,l}(h_1m)(h_{2}f_l) \otimes (h_3.c_l)\\
&=& \sum_l[h_1(mf_l)] \otimes h_2.c_l\\
&=& h(\sum_l mf_l \otimes c_l)\\
&=& h\rho_{M, \mathcal C}(m).\end{array}$$  
So $\rho_{M, \mathcal C}$ is a homomorphism of $H$-modules. Let us equip $M$ with its structure of right $A$-module induced by the algebra homomorphism $i: A \rightarrow {^{\prime}{\mathcal C}}$. By Lemma 2.9(2), $i$ is a homomorphism of $H$-modules. Since the $H$-action on ${^{\prime}{\mathcal C}}$ is compatible with the algebra product of ${^{\prime}{\mathcal C}}$ (Lemma 2.9(3)), we get that
$$h(ma)=(h_1m)(h_2.a).$$ So $M$ is a left-right $(H,A)$-module. It follows that $M$ is an object of $_H{\mathcal M}^{\mathcal C}$. Adapting the proof of \cite[Proposition 38]{caeMilZhu}, we have $(GF)(M)=M$ for every $M$ in $_H{\mathcal M}^{\mathcal C}$ and $(FG)(M)=M$ for every $M$ in $_H{\mathcal M}_{^{\prime}{\mathcal C}}$. 
\end{proof}

\section{The main results}

In this section, the notations and the conventions of the preceding section are still valid, $H$ is a bialgebra, $A$ is an $H$-module algebra  and $\mathcal C$ an $(H, A)$-coring. Given two right $(H,{\mathcal C})$-comodules $\Lambda$ and $N$, the $k$-module ${}_HHom^{\mathcal C}(\Lambda , N)$ is a right module over the endomorphism ring $B={}_HEnd^{\mathcal C}(\Lambda)$ with the standard action    
$$fb=f\circ b; \forall f \in{}_HHom^{\mathcal C}(\Lambda , N), b \in B.$$
We get a functor $G={}_HHom^{\mathcal C}(\Lambda , -): {}_H{\mathcal M}^{\mathcal C} \rightarrow {\mathcal M}_B$. Let us consider $\Lambda$ as a left $B$-module by $b\lambda=b(\lambda)$. So $\Lambda$ is a $(B,A)$-bimodule. For any right
$B$-module $P$, $P \otimes_B \Lambda$ is a right $(H,{\mathcal C})$-comodule: the $H$-module structure is given by $h(p \otimes_B \lambda)=p \otimes_B (h\lambda)$ which is well-defined, while the comodule structure is given by $\rho_{P \otimes_B \Lambda , {\mathcal C}}=id_P \otimes_B \rho_{\Lambda,{\mathcal C}}$. The
functor $G$ has the left adjoint $F=-\otimes_B \Lambda: {\mathcal M}_B \rightarrow {}_H{\mathcal M}^{\mathcal C}$; that is,  for $N \in {}_H{\mathcal M}^{\mathcal C}$ and $P \in {\mathcal M}_B$, there is a canonical isomorphism,
$$_HHom^{\mathcal C}(P \otimes_B \Lambda , N) \rightarrow Hom_B(P,{}_HHom^{\mathcal C}(\Lambda , N)); \quad f \mapsto[p \mapsto f(p \otimes_B -)]$$
with inverse map $g \mapsto [p \otimes_B \lambda \mapsto g(p)(\lambda)]$.

The unit of the adjunction is given by
$$u_N : N \rightarrow {}_HHom^{\mathcal C}(\Lambda , N \otimes_B \Lambda), n
\mapsto [\lambda \mapsto n \otimes_B \lambda]$$ for $N \in {\mathcal M}_B$ while the counit is the evaluation map
$$c_M : {}_HHom^{\mathcal C}(\Lambda , M)\otimes_B \Lambda \rightarrow M, f\otimes_B \lambda \mapsto f(\lambda)$$ for $M \in {}_H{\mathcal M}^{\mathcal C}$.

The adjointness property means that we have
$$G(c_M) \circ u_{G(M)}=id_{G(M)}, \quad c_{F(N)} \circ F(u_N)=id_{F(N)};
M \in {}_H{\mathcal M}^{\mathcal C}, N \in {\mathcal M}_B  \eqno(\star).$$

Let $A$ be an $H$-algebra and $\mathcal C$ an $(H,A)$-coring which is left $A$-flat. An object $\Lambda \in {}_H{\mathcal M}^{\mathcal C}$ is called semi-$\Sigma$-quasi-projective if the functor $_HHom^{\mathcal C}(\Lambda , -) : {}_H{\mathcal M}^{\mathcal C} \rightarrow{\mathcal M}$ sends an exact sequence in $_H{\mathcal M}^{\mathcal C}$ of the form
$$\Lambda^{(J)} \rightarrow \Lambda^{(I)} \rightarrow N \rightarrow 0$$
to an exact sequence \cite{sato}. Obviously, a projective object in $_H{\mathcal M}^{\mathcal C}$ is semi-$\Sigma$-quasi-projective.

\begin{lemma} Assume that $\mathcal C$ is flat as a left $A$-module. Let $\Lambda$ be a right $(H,{\mathcal C})$-comodule that is a finitely generated object in ${}_H{\mathcal M}^{\mathcal C}$, and let $B ={}_HEnd^{\mathcal C}(\Lambda)$. For every index set $I$,

(1) the natural map $\kappa : B^{(I)}={}_HHom^{\mathcal C}(\Lambda , \Lambda)^{(I)} \rightarrow {}_HHom^{\mathcal C}(\Lambda , \Lambda^{(I)})$ is an isomorphism;

(2) $c_{\Lambda^{(I)}}$ is an isomorphism;

(3) $u_{B^{(I)}}$ is an isomorphism;

(4) if $\Lambda$ is semi-$\Sigma$-quasi-projective in $_H{\mathcal M}^{\mathcal C}$, then $u$ is a natural isomorphism; in other words, the induction functor $F =(-) \otimes_B \Lambda$ is fully faithful.

\end{lemma}

\begin{proof} Since $\Lambda$ is a finitely generated object in ${}_H{\mathcal M}^{\mathcal C}$, the functor $_HHom^{\mathcal C}(\Lambda,-)$ commutes with direct sums.

(1) It is easy.

(2) It is straightforward to check that the canonical isomorphism $B^{(I)}\otimes_ B \Lambda \simeq \Lambda^{(I)}$ is nothing else
than $c_{\Lambda^{(I)}} \circ (\kappa \otimes id_{\Lambda})$. It follows from (1) that $\kappa \otimes id_{\Lambda}$ is an
isomorphism. So $c_{\Lambda^{(I)}}$ is an isomorphism.

(3) Putting $M=\Lambda^{(I)}$ in $(\star)$ and using (1), we find
$$_HHom^{\mathcal C}(\Lambda , c_{\Lambda^{(I)}}) \circ
u_{Hom^{\mathcal C}(\Lambda, \Lambda^{(I)})} =id_{_HHom^{\mathcal C}(\Lambda, \Lambda^{(I)})}, i.e.,$$
$$_HHom^{\mathcal C}(\Lambda , c_{\Lambda^{(I)}}) \circ
u_{B^{(I)}}=id_{B^{(I)}}.$$ From (2), $_HHom^{\mathcal C}(\Lambda,
c_{\Lambda^{(I)}})$ is an isomorphism, hence $u_{B^{(I)}}$ is an isomorphism.

(4) Take a free resolution $B^{(J)} \rightarrow B^{(I)} \rightarrow N \rightarrow 0$ of a right $B$-module $N$. Since $u$
is natural, we have a commutative diagram
$$\diagram{\bf B^{(J)} &\hfl{}{}{} &{\bf B^{(I)}} &\hfl{}{}{}&
\quad \quad {\bf N} \hfl{}{}{} & 0 \cr {} u_{B^{(J)}} \kfl{}{}
&&\kfl {u_{B^{(I)}}}{} &&\kfl {u_N}{} && \cr {\bf GF(B^{(J)})}
&\hfl{}{}& {\bf GF(B^{(I)})}& \hfl{}{}& {\bf GF(N)} &\hfl{} {}{}& 0 \cr}$$
The top row is exact. The bottom row is exact, since
$$GF(B^{(I)})={}_HHom^{\mathcal C}(\Lambda, B^{(I)} \otimes_B \Lambda)={}_HHom^{\mathcal C}(\Lambda, \Lambda^{(I)})$$ and $\Lambda$ is semi-$\Sigma$-quasi-projective. By (3), $u_{B^{(I)}}$ and $u_{B^{(J)}}$ are isomorphisms; and it follows from the five lemma that $u_N$ is an isomorphism.
\end{proof}

We can now give equivalent conditions for the projectivity and flatness of $P \in {\mathcal M}_B$.

\begin{theorem} Assume that $\mathcal C$ is flat as a left $A$-module. Let $\Lambda$ be a right $(H,{\mathcal C})$-comodule that is a finitely generated object in ${}_H{\mathcal M}^{\mathcal C}$, and let $B ={}_HEnd^{\mathcal C}(\Lambda)$. For $P \in {\mathcal M}_B$, we consider the following statements.

(1) $P \otimes_B \Lambda$ is projective in $_H{\mathcal M}^{\mathcal C}$ and $u_P$ is injective;

(2) $P$ is projective as a right $B$-module;

(3) $P \otimes_B \Lambda$ is a direct summand in $_H{\mathcal M}^{\mathcal C}$ of some $\Lambda^{(I)}$, and $u_P$ is bijective;

(4) there exists $Q \in {}_H{\mathcal M}^{\mathcal C}$ such that $Q$ is a
direct summand of some $\Lambda^{(I)}$, and $P \cong {}_HHom^{\mathcal C}(\Lambda , Q)$ in ${\mathcal M}_B$;

(5) $P \otimes_B \Lambda$ is a direct summand in $_H{\mathcal M}^{\mathcal C}$ of some $\Lambda^{(I)}$.

Then (1) $\Rightarrow$ (2) $\Leftrightarrow$ (3) $\Leftrightarrow$ (4) $\Rightarrow$ (5).

If $\Lambda$ is semi-$\Sigma$-quasi-projective in $_H{\mathcal M}^{\mathcal C}$, then (5) $\Rightarrow$ (3); if $\Lambda$ is projective in
$_H{\mathcal M}^{\mathcal C}$, then (3) $\Rightarrow$ (1).

\end{theorem}

\begin{proof} (2) $\Rightarrow$ (3). If P is projective as a right $B$-module, then we can find an index set $I$ and $P' \in {\mathcal  
M}_B$ such that $B^{(I)}\cong P \oplus P'$. Then obviously
$$\Lambda^{(I)} \cong B^{(I)} \otimes_B \Lambda \cong (P \otimes_B \Lambda)\oplus (P' \otimes_B \Lambda).$$ 
Since $u$ is a natural transformation, we have a commutative diagram
$$\diagram{\bf B^{(I)} &\hfl {\scriptstyle \cong} {}{}&{\bf P \oplus P'}\cr {u_{B^{(I)}}}\kfl{}{}&& \kfl{\scriptstyle }{u_P
\oplus u_{P'}}&&\cr{\bf _HHom^{\mathcal C}(\Lambda , \Lambda^{(I)})} &\hfl{\scriptstyle} {\cong}&{\bf _HHom^{\mathcal C}(\Lambda ,P \otimes_B \Lambda) \oplus {}_HHom^{\mathcal C}(\Lambda ,P' \otimes_B \Lambda)}\cr}.$$
From the fact that $u_{B^{(I)}}$ is an isomorphism, it follows that $u_P$ (and $u_{P'}$) are isomorphisms.

(3) $\Rightarrow$ (4). Take $Q=P \otimes_B \Lambda$.

(4) $\Rightarrow$ (2). Let $f : \Lambda^{(I)} \rightarrow Q$ be a split epimorphism in $_H{\mathcal M}^{\mathcal C}$. Then

$$_HHom^{\mathcal C}(\Lambda, f) : {}_HHom^{\mathcal C}(\Lambda, \Lambda^{(I)})\cong B^{(I)} \rightarrow  {}_HHom^{\mathcal C}(\Lambda, Q) \cong P$$ is also split surjective, hence $P$ is projective as a right $B$-module.

(4) $\Rightarrow$ (5). If (4) is true, we know from the proof of (4) $\Rightarrow$ (2) that $P$ is a direct summand of some
$B^{(I)}$. So $P \otimes_B \Lambda$ is direct summand of $\Lambda ^{(I)}$.

Under the assumption that $\Lambda$ is semi-$\Sigma$-quasi-projective in $_H{\mathcal M}^{\mathcal C}$, (5) $\Rightarrow$ (3) follows from Lemma 3.1(4).

(1) $\Rightarrow$ (2). Take an epimorphism $f : B^{(I)} \rightarrow P$ in ${\mathcal M}_B$. Then
$$F(f)=f \otimes_B id_{\Lambda} : B^{(I)} \otimes_B \Lambda \cong \Lambda^{(I)} \rightarrow P \otimes_B \Lambda$$ 
is also surjective, and split in $_H{\mathcal M}^{\mathcal C}$ since $P \otimes_B \Lambda$ is projective. Consider the commutative diagram
$$\diagram{\bf B^{(I)} &\hfl {\scriptstyle f} {}{}&{\bf P}&\hfl {\scriptstyle} {}{}& 0\cr {} u_{B^{(I)}}\kfl{}{}&&
\kfl{\scriptstyle u_P}{}&&\cr{\bf _HHom^{\mathcal C}(\Lambda , \Lambda^{(I)})} &\hfl{ \scriptstyle} {GF(f)}&{\bf _HHom^{\mathcal C}(\Lambda , P \otimes_B \Lambda)}&\hfl {\scriptstyle} {}{}& 0\cr}$$ The bottom row is split exact, since any functor, in particular $_HHom^{\mathcal C}(\Lambda, -)$ preserves split exact sequences. By Lemma 3.1(3), $u_{B^{(I)}}$ is an isomorphism. A diagram chasing tells us that $u_P$ is surjective. By assumption, $u_P$ is injective, so $u_P$ is bijective. We deduce that the top
row is isomorphic to the bottom row, and therefore splits. Thus $P \in {\mathcal M}_B$ is projective.

(3) $\Rightarrow$ (1). By (3), $P \otimes_B \Lambda$ is a direct summand of some $\Lambda^{(I)}$. If $\Lambda$ is projective in
$_H{\mathcal M}^{\mathcal C}$, then $\Lambda^{(I)}$ is projective in $_H{\mathcal M}^{\mathcal C}$. So $P \otimes_B \Lambda$ being a direct summand of a projective object of $_H{\mathcal M}^{\mathcal C}$ is projective in $_H{\mathcal M}^{\mathcal C}$.
\end{proof}

\begin{theorem} Assume that $\mathcal C$ is flat as a left $A$-module. Let $\Lambda$ be a right $(H,{\mathcal C})$-comodule that is a finitely presented object in ${}_H{\mathcal M}^{\mathcal C}$, and let $B = {}_HEnd^{\mathcal C}(\Lambda)$. For $P \in {\mathcal M}_B$, the following assertions are equivalent.

(1) $P$ is flat as a right $B$-module;

(2) $P \otimes_B \Lambda=\limind Q_i$, where $Q_i \cong \Lambda^{n_i}$ in $_H{\mathcal M}^{\mathcal C}$ for some positive integer $n_i$, and $u_P$ is bijective;

(3) $P \otimes_B \Lambda=\limind Q_i$, where $Q_i \in {}_H{\mathcal M}^{\mathcal C}$ is a direct summand of some $\Lambda^{(I_i)}$, and $u_P$ is bijective;

(4) there exists $Q=\limind Q_i \in {}_H{\mathcal M}^{\mathcal C}$, such that $Q_i \cong \Lambda^{n_i}$ for some positive integer $n_i$ and $_HHom^{\mathcal C}(\Lambda, Q) \cong P$ in ${\mathcal M}_B$;
    
(5) there exists $Q=\limind Q_i \in {}_H{\mathcal M}^{\mathcal C}$, such that $Q_i$ is a direct summand of some $\Lambda^{(I_i)}$ in $_H{\mathcal M}^{\mathcal C}$, and $_HHom^{\mathcal C}(\Lambda, Q) \cong P$ in ${\mathcal M}_B$.

If $\Lambda$ is semi-$\Sigma$-quasi-projective in $_H{\mathcal M}^{\mathcal C}$, these conditions are also equivalent without the assumption that $u_P$ is bijective in (2) and (3).
 
\end{theorem}

\begin{proof} Since $\Lambda$ is a finitely presented object in ${}_H{\mathcal M}^{\mathcal C}$, the functor $_HHom^{\mathcal C}(\Lambda,-)$ commutes with direct limits.

(1) $\Rightarrow$ (2) $P = \limind N_i$, with $N_i = B^{n_i}$. Take $Q_i =\Lambda^{n_i}$, then $$\limind Q_i
\cong \limind (N_i \otimes_B \Lambda) \cong (\limind N_i) \otimes_B \Lambda \cong P \otimes_B \Lambda.$$ Consider the
following commutative diagram:
$$\diagram{\bf P=\limind N_i &\hfl
{\scriptstyle \lim(u_{N_i})} {}{}&{\bf \limind {}_HHom^{\mathcal C}(\Lambda , N_i \otimes_B\Lambda)}\cr {u_P}\kfl{}{}&&
\kfl{\scriptstyle }{f}&&\cr{\bf {}_HHom^{\mathcal C}(\Lambda , (\limind
N_i) \otimes_B \Lambda)} &\hfl{\scriptstyle} {\cong}&{\bf
_HHom^{\mathcal C}(\Lambda , \limind (N_i \otimes_B \Lambda))}.\cr}$$ 
By Lemma 3.1(3), the $u_{N_i}$ are isomorphisms. The natural homomorphism $f$ is an isomorphism because $\Lambda$ is a finitely
presented object in $_H{\mathcal M}^{\mathcal C}$. Hence $u_P$ is an isomorphism.

(2) $\Rightarrow$ (3) and (4) $\Rightarrow$ (5) are obvious.

(2) $\Rightarrow$ (4) and (3) $\Rightarrow$ (5) Put $Q=P \otimes_B \Lambda$. Then $u_P : P \rightarrow {}_HHom^{\mathcal C}(\Lambda, P \otimes_B \Lambda)$ is the required isomorphism.

(5) $\Rightarrow$ (1). We have a split exact sequence $0 \rightarrow  N_i \rightarrow P_i=\Lambda^{(I_i)} \rightarrow Q_i
\rightarrow 0$. Consider the following commutative diagram
$$\diagram{{\bf0}&\hfl{ }{}&{\bf FG(N_i)}\hfl{ }{}&{\bf FG(P_i)}\qquad \hfl{ }{}&&{\bf FG(Q_i)} &\hfl{} {} {\bf0} &\cr
&&\kfl{c_{N_i}}{}\quad\qquad \quad\qquad\quad&\kfl{c_{P_i}}{}
\qquad &&\kfl{c_{Q_i}}{}\cr {\bf0} &\hfl{ }{}&{\bf N_i}\qquad
\hfl{ }{}& {\bf P_i}\qquad \hfl{ }{}&& {\bf Q_i} &\hfl{} {}
{\bf0}.\cr }$$ 
We know from Lemma 3.1(2) that $c_{P_i}$ is an isomorphism. Both rows in the diagram are split exact, so it follows that $c_{N_i}$ and $c_{Q_i}$ are also isomorphisms. Next, consider the commutative diagram
$$\diagram{{\bf (\limind {}_HHom^{\mathcal C}(\Lambda , Q_i)) \otimes_B \Lambda} &
\hfl {\scriptstyle f \otimes id_{\Lambda}}{}{}&{\bf {}_HHom^{\mathcal C}(\Lambda , Q) \otimes_B \Lambda} \cr h\vfl{}{}&&
\kfl{\scriptstyle c_Q}{}&&\cr {\bf \limind(_HHom^{\mathcal C}(\Lambda ,Q_i) \otimes_B \Lambda)}&\hfl{\scriptstyle} {\lim c_{Q_i}}&{\bf Q},\cr}$$ where $h$ and $f$ are the natural homomorphisms. $h$ is an isomorphism because the functor $(-) \otimes_B \Lambda$ preserves inductive limits; $f$ is an isomorphism because $\Lambda$ is a finitely presented object in $_H{\mathcal M}^{\mathcal C}$, and $lim c_{Q_i}$ is an isomorphism because every $c_{Q_i}$ is an isomorphism. It follows that $c_Q$ is an isomorphism. Hence $_HHom^{\mathcal C}(\Lambda , c_Q)$ is an isomorphism. From $(\star)$, we get
$$_HHom^{\mathcal C}(\Lambda , c_Q) \circ u_{_HHom^{\mathcal C}(\Lambda ,Q)}=
id_{_HHom^{\mathcal C}(\Lambda , Q)}.$$ It follows that $u_{_HHom^{\mathcal C}(\Lambda ,Q)}$ is also an isomorphism. Since $_HHom^{\mathcal C}(\Lambda ,Q) \cong P$, $u_P$ is an isomorphism. Consider the isomorphisms
$$P \cong {}_HHom^{\mathcal C}(\Lambda , P \otimes_B \Lambda) \cong {}_HHom^{\mathcal C}(\Lambda , {}_HHom^{\mathcal C}(\Lambda , Q) \otimes_B \Lambda) \cong $$
$$_HHom^{\mathcal C}(\Lambda , Q) \cong
\limind {}_HHom^{\mathcal C}(\Lambda , Q_i);$$ where the first isomorphism is $u_P$, the third is $_HHom^{\mathcal C}(\Lambda , c_Q)$ and the last one is $f$. It follows from Lemma 3.1(1) that $_HHom^{\mathcal C}(\Lambda , P_i) \cong B^{(I_i)}$ is projective as a right $B$-module, hence $_HHom^{\mathcal C}(\Lambda , Q_i)$ is also projective as a right $B$-module, since $Q_i$ is a direct summand of $P_i$. We conclude that $P \in {\mathcal M}_B$ is flat.

The final statement is an immediate consequence of Lemma 3.1(4).
\end{proof}

We get from the main theorems the following results.	
	
\begin{remark} Assume that $H$ is a Hopf algebra with a bijective antipode, and let $\mathcal C$ be an $(H,A)$-coring which is projective and finitely generated as a left $A$-module. By Lemma 2.11, 
$_H{\mathcal M}^{\mathcal C}={}_H{\mathcal M}_{^{\prime}{\mathcal C}}={}_{(^{\prime}{\mathcal C})^{op} \#H}{\mathcal M}$. For every left $(^{\prime}{\mathcal C})^{op} \#H$-module $\Lambda$, we have 
$_HEnd^{\mathcal C}(\Lambda)={}_{(^{\prime}{\mathcal C})^{op} \# H}End(\Lambda)$, the endomorphism ring of the left $(^{\prime}{\mathcal C})^{op} \#H$-module $\Lambda$. Theorems 3.2 and 3.3 give necessary and sufficient conditions for projectivity and flatness of a module over $_{(^{\prime}{\mathcal C})^{op} \# H}End(\Lambda)$, when $\Lambda$ is a finitely generated (finitely presented) left $(^{\prime}{\mathcal C})^{op} \#H$-module. Taking $H=k$ with a trivial Hopf algebra structure, we recover \cite[Theorems 2.2 and 2.3]{CG}.
\end{remark}

\section{The $(H,A)$-coring contains a fixed $H$-grouplike element}
We keep the notations and the conventions of the preceding sections. $H$ is a bialgebra, $A$ is an $H$-module algebra and $\mathcal C$ is an $(H,A)$-coring.

\begin{definition} An $H$-grouplike element of ${\mathcal C}$ is an element $x\in \mathcal C$ such that $\Delta_{\mathcal C}(x)=x\otimes_Ax$, $\epsilon_{\mathcal C}(x)=1_A$ and $h.x=\epsilon_H(h)x$ for all $h \in H$, i.e., an $H$-grouplike element of ${\mathcal C}$ is a grouplike element $x$ of $\mathcal C$ which is $H$-invariant.
\end{definition}

For the remainder of the section, ${\mathcal C}$ contains a fixed $H$-grouplike element $x$. 

\begin{lemma}
(1) $A$ is an object of $_H{\mathcal M}^{\mathcal C}$: the right $A$-module and the right $H$-module structures are defined naturally and the $\mathcal C$-coaction is defined by 
$$\rho_{A, \mathcal C}(a)=1_A \otimes_A (x \blacktriangleleft a)=x \blacktriangleleft a; \quad \forall a \in A.$$

(2) $\rho_{A, \mathcal C}(1_A)$ is an $H$-grouplike element of ${\mathcal C}$.

\end{lemma} 

For any right $(H,{\mathcal C})$-comodule $M$, the vector space
$$M^{H,co{\mathcal C}, x}=\{ m \in M, \quad {\rho}_{M , {\mathcal C}}(m)=m \otimes_A x \quad \hbox{and} \quad hm =\epsilon_H(h)m \quad \forall \quad h \in H\}$$ is called the $k$-submodule of $(H,{\mathcal C} , x)$-coinvariants of $M$.

We have $$A^{H,co{\mathcal C}, x}=\{a \in A, \quad x \blacktriangleleft a=a\blacktriangleright x \quad \hbox{and} \quad h.a=\epsilon_H(h)a \quad \forall \quad h \in H\}.$$

\begin{lemma} $A^{H,co{\mathcal C}, x}$ is an $H$-subalgebra of $A$ called the subalgebra of $(H,{\mathcal C},x)$-coinvariants of $A$.
\end{lemma}

Let $M$ be a right $(H,{\mathcal C})$-comodule. We can show that $M^{H,co{\mathcal C}, x}$ and $_HHom^{\mathcal C}(A , M)$ are right $A^{H,co {\mathcal C}, x}$-modules: the $A^{H,co {\mathcal C}, x}$-action on $_HHom^{\mathcal C}(A , M)$ is defined by $(fb)(a)=f(ba)$ for every $f \in {}_HHom^{\mathcal C}(A , M)$, $b \in A^{H,co {\mathcal C}, x}$ and $a \in A$. We can also show that the map
$$\psi: {}_HHom^{\mathcal C}(A , M) \rightarrow M^{H,co{\mathcal C}, x}, \quad f \mapsto f(1_A)$$
is an isomorphism of right $A^{H,co {\mathcal C}, x}$-modules. Furthermore, the algebras ${}_HEnd^{\mathcal C}(A)$ and $A^{H,co {\mathcal C},x}$ are isomorphic.

\begin{lemma} Set $B=A^{H,co {\mathcal C}, x}$. Let $M$ be a right $B$-module.

(1) If $N$ is an $H$-module which is also a $(B,A)$-bimodule (for example $N=A$), then $M \otimes_B N$ is a left-right $(H,A)$-module: the $A$-action is given by $(m \otimes_B n)a=m \otimes_B (na)$, the $H$-action is given by $h(m \otimes_B n)=m \otimes_B (hn)$. 

(2) If $N$ is a right $(H,{\mathcal C})$-comodule which is also a $(B,A)$-bimodule such that the $\mathcal C$-coaction is left $B$-linear (for example, $N=A$), then $M \otimes_B N$ is a left-right $(H,{\mathcal C})$-comodule: the $A$-action is given by $(m \otimes_B n)a=m \otimes_B (na)$, the $H$-action is given by $h(m \otimes_B n)=m \otimes_B (hn)$ and the coaction is given by $\rho_{M \otimes_BN,{\mathcal C}}(m \otimes_B n)=(m \otimes_B n_0) \otimes_A n_1$.

\end{lemma}

\begin{proof} (1) Clearly, $M \otimes_B N$ is a right $A$-module. Note that the $H$-action is well defined because $B$ is a trivial $H$-module. It is easy to see that $M \otimes_B N$ is an $H$-module for the given structure.
We have 
$$\begin{array}{rcl} h[(m \otimes_B n)a]&=& h(m \otimes_B (na))\\
&=& m \otimes_B [h(na)]\\
&=& m \otimes_B (h_1n)(h_2.a)\\
&=& (m \otimes_B (h_1n))(h_2.a)\\
&=& h_1(m \otimes_B n)(h_2.a).\end{array}$$
So  $M \otimes_B N$ is a left-right $(H,A)$-module.

(2) Since the $\mathcal C$-coaction of $N$ is left $B$-linear, the $\mathcal C$-coaction of $M \otimes_BN$ is well defined. We  have
$$\begin{array}{rcl}\rho_{M \otimes_BA,{\mathcal C}}(h[m \otimes_B n]) &=& \rho_{M \otimes_BA,\mathcal C}(m \otimes_B(hn))\\
&=& m \otimes_B(hn)_0 \otimes_A (hn)_1\\
&=& m \otimes_B(h_1n_0) \otimes_A (h_2n _1)\\
&=& m \otimes_Bh(n_0 \otimes_A n _1)\\
&=& h(m \otimes_B n_0 \otimes_A n _1)\\
&=& h[\rho_{M \otimes_BN,{\mathcal C}}(m \otimes_B n)].\end{array},$$
thus $\rho_{M \otimes_BA,{\mathcal C}}$ is $H$-linear. It is easy to show that $\rho_{M \otimes_BA,{\mathcal C}}$ is right $A$-linear.
\end{proof}

Let $\mathcal C$ be an $(H,A)$-coring containing a fixed $H$-grouplike element $x$. Set $B=A^{H,co{\mathcal C}, x}$. By Lemme 4.2, $A$ is an object of $_H{\mathcal M}^{\mathcal C}$.
Then replacing $\Lambda$ with $A$, the functors $F$ and $G$ introduced in Section 3 become
$$F=(-) \otimes_B A : {\mathcal M}_B \rightarrow {}_H{\mathcal M}^{{\mathcal C}}; \quad N \mapsto N \otimes_BA$$
$$G=(-)^{H,co{\mathcal C}, x}: {}_H{\mathcal M}^{{\mathcal C}} \rightarrow {\mathcal M}_B; \quad M \mapsto M^{H,co{\mathcal C}, x}.$$
The unit of the adjunction pair $(F,G)$ is given by
$$u_N : N \rightarrow (N \otimes_B A)^{H,co{\mathcal C} , x}; n \mapsto (n \otimes_B 1_A)$$
for every right $B$-module $N$ while the counit is
$$c_M : M^{H,co{\mathcal C}, x} \otimes_B A \rightarrow M; m \otimes_B a \mapsto ma$$
for every right $(H,{\mathcal C})$-comodule $M$. Let $\mathcal C$ be flat as left $A$-module. We deduce from Theorems 3.2 and 3.3 necessary and sufficient conditions for projectivity and flatness of a module over $B$ if $A$ is finitely generated (finitely presented) in $_H{\mathcal M}^{\mathcal C}$. 

Until the end of the section, $H$ is a Hopf algebra with a bijective antipode.

\begin{lemma} $A$ is an object of $_H{\mathcal M}^{\mathcal C}$ and a cyclic left $(^{\prime}{\mathcal C})^{op} \# H$-module: the right ${^{\prime}{\mathcal C}}$-module action on $A$ is defined by $a \leftharpoonup f=f(x \blacktriangleleft a)$ for all $f \in {^{\prime}{\mathcal C}}$ and $a \in A$.
\end{lemma}

\begin{proof} Note that $$a \leftharpoonup \epsilon_{\mathcal C}=\epsilon_{\mathcal C}(x\blacktriangleleft a)=\epsilon_{\mathcal C}(x)a=1_Aa=a$$ and
$$\Delta_{\mathcal C}(x \blacktriangleleft a)=x \otimes_A (x \blacktriangleleft a).$$
Now for $f , g \in {}^{\prime}{\mathcal C}$, we have
$$a \leftharpoonup (f\#g)=(f\#g)(x \blacktriangleleft a)=g(x \blacktriangleleft [f(x \blacktriangleleft a)])=(a \leftharpoonup f) \leftharpoonup g.$$
Then $A$ is a right $^{\prime}{\mathcal C}$-module. Clearly, $A$ is an $H$-module. By Lemma 1.9, we have 
$$(hf)(c)=h_2f(S_H^{-1}(h_1).c) \quad \forall h \in H, \quad c \in {\mathcal C} \quad f \in {^{\prime}{\mathcal C}}.$$
So we have
$$\begin{array}{rcl}(h_1.a) \leftharpoonup (h_2f) &=& (h_2f)(x \blacktriangleleft (h_1.a))\\
&=& h_{3}f[S_H^{-1}(h_{2}).(x \blacktriangleleft (h_1.a))]\\
&=& h_{4}f[(S_H^{-1}(h_{3}).x) \blacktriangleleft (S_H^{-1}(h_{2})(h_1.a))]\\
&=& h_{3}f[x \blacktriangleleft S_H^{-1}(h_{2})(h_1.a)]\\
&=& h_{2}f[x \blacktriangleleft \epsilon_H(h_1)a]\\
&=& hf(x \blacktriangleleft a)\\
&=& h(a \leftharpoonup f).
\end{array}$$
So $A$ is a left-right $(H,^{\prime}{\mathcal C})$-module. Therefore, $A$ is a left $(^{\prime}{\mathcal C})^{op} \# H$-module. Now we have $a=1_A \leftharpoonup i(a)$ for all $a \in A$, where $i$ is the $H$-module homomorphism defined from $A$ to  $^{\prime}{\mathcal C}$ by $i(a)(c)=\epsilon_{\mathcal C}(c)a$. So $A$ is generated as a right $^{\prime}{\mathcal C}$-module by $1_A$; or equivalently, $A$ is generated as a left $(^{\prime}{\mathcal C})^{op}$-module by $1_A$. Since every element of $(^{\prime}{\mathcal C})^{op}$ can be identified with an element of $(^{\prime}{\mathcal C})^{op} \# H$ and since $i(a)$ is an element of $^{\prime}{\mathcal C}$, we conclude that $A$ is generated as a left $(^{\prime}{\mathcal C})^{op} \# H$-module by $1_A$.
\end{proof}

From Theorem 3.2, we obtain the following result.

\begin{theorem} Assume that $\mathcal C$ is projective as a left $A$-module and contains a fixed $H$-grouplike element $x$. Set 
$B=A^{H,co{\mathcal C}, x}$. For $P \in {\mathcal M}_B$, we consider the following statements.

(1) $P \otimes_B A$ is projective in $_H{\mathcal M}^{\mathcal C}$ and $u_P$ is injective;

(2) $P$ is projective as a right $B$-module;

(3) $P \otimes_B A$ is a direct summand in $_H{\mathcal M}^{\mathcal C}$ of some $A^{(I)}$, and $u_P$ is bijective;

(4) there exists $Q \in{}_H{\mathcal M}^{\mathcal C}$ such that $Q$ is a direct summand of some $A^{(I)}$, and $P \cong {}_HHom^{\mathcal C}(A ,
Q)$ in ${\mathcal M}_B$;

(5) $P \otimes_B A$ is a direct summand in $_H{\mathcal M}^{\mathcal C}$ of some $A^{(I)}$.

Then (1) $\Rightarrow$ (2) $\Leftrightarrow$ (3) $\Leftrightarrow$ (4) $\Rightarrow$ (5).

If $A$ is semi-$\Sigma$-quasi-projective in $_H{\mathcal M}^{{\mathcal C}}$, then (5) $\Rightarrow$ (3). If $A$ is projective in $_H{\mathcal M}^{{\mathcal C}}$, then (3) $\Rightarrow$ (1).

\end{theorem}

\begin{proof} By Lemma 4.5, $A$ is an object of $_H{\mathcal M}^{{\mathcal C}}$ that is finitely generated as a left $(^{\prime}{\mathcal C})^{op} \# H$-module. So $A$ is finitely generated in $_H{\mathcal M}^{{\mathcal C}}$. 
\end{proof}

From Theorem 3.3, we obtain the following result.

\begin{theorem} Assume that $\mathcal C$ is projective as a left $A$-module, contains a fixed $H$-grouplike element $x$, and that $A$ is finitely presented as a left $(^{\prime}{\mathcal C})^{op} \#H$-module. Set $B=A^{H,co{\mathcal C}, x}$. For $P \in {\mathcal M}_B$, the following assertions are equivalent.

(1) $P$ is flat as a right $B$-module;

(2) $P \otimes_B A=\limind Q_i$, where $Q_i \cong A^{n_i}$ in
$_H{\mathcal M}^{\mathcal C}$ for some positive integer $n_i$, and $u_P$ is bijective;

(3) $P \otimes_B A=\limind Q_i$, where $Q_i \in {}_H{\mathcal M}^{\mathcal C}$ is a direct summand of some $A^{(I_i)}$ in $_H{\mathcal M}^{\mathcal C}$,
and $u_P$ is bijective;

(4) there exists $Q=\limind Q_i \in {}_H{\mathcal M}^{\mathcal C}$, such that $Q_i \cong A^{n_i}$ for some positive integer $n_i$ and $_HHom^{\mathcal C}(A, Q) \cong P$ in ${\mathcal M}_B$;

(5) there exists $Q=\limind Q_i \in {}_H{\mathcal M}^{\mathcal C}$, such that $Q_i$ is a direct summand of some $A^{(I_i)}$ in $_H{\mathcal M}^{{\mathcal C}}$, and $_HHom^{\mathcal C}(A, Q) \cong P$ in ${\mathcal M}_B$.

\bigskip

If $A$ is semi-$\Sigma$-quasi-projective in $_H{\mathcal M}^{\mathcal C}$, these conditions are also equivalent without the assumption that $u_P$ is bijective in conditions (2) and (3).

\end{theorem}

\begin{proof} $A$ is an object of $_H{\mathcal M}^{{\mathcal C}}$ that is finitely presented as a left $(^{\prime}{\mathcal C})^{op} \# H$-module. So $A$ is finitely presented in $_H{\mathcal M}^{{\mathcal C}}$. 
\end{proof}

The following lemma will enable us to improve Theorems 4.6 and 4.7 if $P$ is a left-right $(H,A)$-module, a particular $B$-module.

\begin{lemma} Set $B=A^{H,co{\mathcal C}, x}$. Let $N$ be a left-right $(H,A)$-module. Then $u_N$ is an injection of right $B$-modules.
\end{lemma}

\begin{proof} We know that $\mathcal C$ is a right $(H,{\mathcal C})$-comodule with coaction $\Delta_{\mathcal C}$, and $N \otimes_A {\mathcal C}$ is a right $(H,{\mathcal C})$-comodule with coaction $id_N \otimes \Delta_{\mathcal C}$. By the proof of \cite[18.10(1)]{brWi}, the $k$-linear map
$$\varphi : (N \otimes_A {\mathcal C})^{H,co{\mathcal C}, x} \rightarrow N; n \otimes c \mapsto n \epsilon_{\mathcal C}(c)$$ is an isomorphism of right $B$-modules; i.e., $G(W) \cong N$, where $W=N \otimes_A {\mathcal C}$. Replace $M$ with $W$ in $(\star)$, then we get $G(c_W) \circ u_N =id_N$. So $u_N$ is an injection.
\end{proof}

\begin{remark} By Lemma 4.8, $u_P$ is an injection for every left-right $(H,A)$-module $P$. So in the particular case, where $P$ is a left- right $(H,A)$-module, the injection assumption in Theorem 4.6 is superflous. Therefore we can replace ``bijection'' with ``surjection'' in Theorems 4.6 and 4.7. In other words, we get necessary and sufficient conditions for a left-right $(H,A)$-module $P$ to be projective (resp. flat) as a right $B$-module.
\end{remark}

\section{Some examples}

\subsection{An $H$-algebra as an $(H,A)$-coring}
We keep the notations and conventions of the preceding sections. $H$ is a bialgebra and $A$ is an $H$-algebra. We know that $A$ is a left $(H,A)$-bimodule. Let us define $\Delta_A(a)=a \otimes_A 1_A$ and $\epsilon_A(a)=a$. Then $A$ is an  $(H,A)$-coring  projective as a left $A$-module. It is well known that the right $A$-comodules are the right $A$-modules. We deduce that the right $(H,A)$-comodules are the left-right $(H,A)$-modules. Then Theorems 3.2 and 3.3 give necessary and sufficient conditions for projectivity and flatness over the endomorphism ring $_HEnd^{A}(\Lambda)={}_HEnd_{A}(\Lambda)$, where $\Lambda$ is a left-right $(H,A)$-module finitely generated (finitely presented) in $_H{\mathcal M}_A={}_{A^{op} \#H}{\mathcal M}$. If $H=k$ is considered as a trivial bialgebra, then we recover the results of \cite{TG}.

Now assume that $H$ is a Hopf algebra with bijective antipode. Let us consider the left dual $^{\prime}A$ of $A$. Its product is defined by $f\#g=g \circ f$. Note that $(g \circ f)(a)=f(a)g(a)$ for every $a \in A$. Clearly, $1_A$ is an $H$-grouplike element of $A$. By Lemma 2.9, $^{\prime}A$ is an $H$-algebra and the smash product $^{\prime}A \# H$ is defined. Recall that the $H$-action on $^{\prime}A$ is defined by 
$$(hf)(a)=h_2f(S_H^{-1}(h_1).a)=a(hf(1_A)) \quad \forall a \in A, h \in H \quad \hbox{and} \quad f \in {^{\prime}A}.$$
We have an algebra homomorphism $i: A \rightarrow {^{\prime}A}$ defined by $i(a)(a')=a'a$ for all $a,a' \in A$. It is easy to see that $i$ is an isomorphism of $H$-algebras: its inverse is defined by $i^{-1}(f)=f(1_A)$ for all $f$ in $^{\prime}A$. Thus we have $(^{\prime}A)^{op} \# H= A^{op} \#H$. Note that $_HEnd_{A}(\Lambda )=End_{A}(\Lambda)^H$, the $H$-invariant elements of $End_{A}(\Lambda)$: the $H$-action of $End_{A}(\Lambda)$ being $(hf)(\lambda)=h_1f(S_H(h_2) \lambda)$. Thus if $\Lambda=A$, then $_HEnd_{A}(A)=A^H$ the subring of $H$-invariants of $A$: we recover the results of \cite{Guedenon 1}. 

Let us give an example of a left-right $(H,A)$-module which is finitely generated as a left $A^{op}\#H$-module.

\begin{lemma} Let $H$ be a bialgebra. Let $\Lambda$ and $A$ be $H$-algebras, $i: \Lambda \rightarrow A$ a homomorphism of $H$-algebras. Then $A$ is a left $(H, \Lambda)$-bimodule with the structures given by $\lambda a =i(\lambda)a$ and $a\lambda=ai(\lambda)$ for all $\lambda \in \Lambda$ and $a \in A$.
\end{lemma}

\begin{proof} We have
$$h.(\lambda a)=h(i(\lambda)a)=(h_1.i(\lambda)(h_2.a)= i(h_1.\lambda)(h_2.a)= (h_1. \lambda)(h_2.a).$$
We also have
$$h.(a\lambda)=h.(ai(\lambda))= (h_1.a)(h_2.i(\lambda))=(h_1.a) i(h_2.\lambda)=(h_1.a)(h_2.\lambda).$$
\end{proof}

\begin{definition} (See \cite{CaeVerShu} section 2, or \cite{CG}) Let $H$ be a bialgebra. Let $\Lambda$ and $A$ be $H$-algebras, $i: \Lambda \rightarrow A$ a homomorphism of $H$-algebras. A map $\chi : A \rightarrow \Lambda$ is called a right $H$-grouplike character on $A$ if $\chi$ is a homomorphism of left-right $(H,\Lambda)$-modules and
$$\chi(\chi(a')a)=\chi(a'a), \quad \hbox{and} \quad \chi(1_A)=1_{\Lambda} \quad \forall \quad a, a' \in A.$$
We then say that $A$ is an $(H,\Lambda)$-ring with a right $H$-grouplike character $\chi$.
\end{definition}

\begin{lemma} Let $H$ be a bialgebra. Let $\Lambda$ and $A$ be $H$-algebras. If $A$ is an $(H,\Lambda)$-ring with a right $H$-grouplike character $\chi$, then $\Lambda$ is a left-right $(H,A)$-module with the right $A$-action given by $\lambda \leftharpoonup a=\chi(\lambda a)$. Furthermore, $\Lambda$ is cyclic as a right $A$-module and as a left $A^{op} \#H$-module. 
\end{lemma}

\begin{proof} By \cite{CaeVerShu} or \cite{CG}, $\Lambda$ is a right $A$-module. We have
$$ \begin{array}{rcl} h( \lambda \leftharpoonup a)&=& h\chi(\lambda a)=\chi(h(\lambda a))=\chi(h(i(\lambda)a))\\
&=&\chi((h_1.i(\lambda))(h_2.a))=\chi(i(h_1.\lambda)(h_2.a))\\
&=&\chi((h_1.\lambda)(h_2.a))=(h_1.\lambda) \leftharpoonup (h_2.a),\end{array}$$
therefore $\Lambda$ is a left-right $(H,A)$-module. 
It follows that $\Lambda$ is a left $A^{op}\#H$-module. Let $\lambda \in \Lambda$. We have
$$1_{\Lambda} \leftharpoonup (1_A \lambda)=\chi(1_A1_{\Lambda}\lambda)=\chi(1_A)1_{\Lambda}\lambda=\lambda.$$
So $\Lambda$ is generated by $1_{\Lambda}$ as a right $A$-module; or equivalently, $\Lambda$ is generated by $1_{\Lambda}$ as a left $A^{op}$-module. Since $A^{op}$ is a subalgebra of $A^{op} \# H$ and $1_A\lambda$ is an element of $A=A^{op}$, we conclude that $\Lambda$ is generated by $1_{\Lambda}$ over $A^{op} \# H$. 
\end{proof}

Then we get necessary and sufficient conditions for projectivity over the endomorphism ring $$_HEnd^{A}(\Lambda)={}_HEnd_A({\Lambda})={}_{A^{op} \# H}End(\Lambda),$$ where $\Lambda$ and $A$ are $H$-algebras and $A$ is an $(H, \Lambda)$-ring with a right $H$-grouplike character. 

Let $H$ be a Hopf algebra with a bijective antipode. Let $\mathcal C$ be an $(H,A)$-coring with a fixed $H$-grouplike element $x$. The left dual $^{\prime}{\mathcal C}$ of $\mathcal C$  is an $(H,A)$-ring with a right $H$-grouplike character $\chi: {}^{\prime}{\mathcal C} \rightarrow A$; $f \mapsto \chi(f)=f(x)$: clearly, $\chi$ is right $A$-linear, $H$-linear, and for all $f,g \in {}^{\prime}{\mathcal C}$, we have
$$\begin{array}{rcl} \chi(\chi(f)g)&=& (\chi(f)g)(x)= g(x \blacktriangleleft \chi(f))\\
&=& g(x \blacktriangleleft f(x))= (f \# g)(x) \\
&=& \chi(f \# g).\end{array}$$
From Lemma 4.5, we know that $A$ is a cyclic $(^{\prime}{\mathcal C})^{op} \# H$-module. If $\mathcal C$ is projective as left $A$-module, Theorem 3.2 gives necessary and sufficient conditions for projectivity over the endomorphism ring $$_HEnd^{\mathcal C}(A)={}_HEnd_{^{\prime}{\mathcal C}}(A)={}_{{^{\prime}{\mathcal C}}^{op} \# H}End(A).$$

\subsection{$k$ is a trivial bialgebra}

Let us consider $k$ as a trivial bialgebra, and let $A$ be a $k$-algebra. Every right $A$-module is a left-right $(k,A)$-module. In the same way, every $A$-bimodule is a left $(k,A)$-bimodule. It follows that an $A$-coring $\mathcal C$ is nothing else a $(k,A)$-coring $\mathcal C$ and a right $\mathcal C$-comodule is nothing else a right $(k,{\mathcal C})$-comodule. If $\mathcal C$ is flat as left $A$-module, then Theorems 3.2 and 3.3 give necessary and sufficient conditions for projectivity and flatness over the endomorphism ring $_kEnd^{\mathcal C }(\Lambda)=End^{\mathcal C}(\Lambda)$, where $\Lambda$ is a right $\mathcal C$-comodule finitely generated (finitely presented) in ${\mathcal M}^{\mathcal C}$. So we recover the results of \cite{TG}.

\subsection{A coalgebra as an $(H,k)$-coring} 

Let $H$ be a bialgebra over $k$. Let $C$ be a $k$-coalgebra which is a left $H$-module such that $\Delta_C$ and $\epsilon_C$ are homomorphisms of left $H$-modules, that is, a coalgebra in the monoidal category of left $H$-modules. Then $C$ is an $(H,k)$-coring, where $k$ is considered as a trivial left $H$-module algebra. If $C$ is flat as left $k$-module, Theorems 3.2 and 3.3 give necessary and sufficient conditions for projectivity and flatness over the endomorphism ring $_{H}End^{C}(\Lambda)$, where $\Lambda$ is a right $(H,C)$-comodule that is a finitely generated (finitely presented) object in $_H{\mathcal M}^C$.  

Note that if $H$ is a Hopf algebra, the category of right $(H,C)$-comodules is the usual category of left-right (H,C)-Hopf modules \cite{Doi}.

\subsection{ Noncommutative Poisson algebras}
Let $k$ be a field, $A$ a noncommutative Poisson algebra and $U(A)$ the universal enveloping algebra of $A$. Then $A$ is a left $U(A)$-module algebra and a right $U(A)$-module algebra since the category of left $U(A)$-modules is isomorphic to the category of right $U(A)$-modules. We can define the notions of a $(U(A),A)$-coring $\mathcal C$ and of a right $(U(A),{\mathcal C})$-comodule. A vector space is a left $(U(A),A)$-bimodule (resp. a left-right $(U(A),A)$-module) if and only if it is a quasi-Poisson $A$-module (resp. a right quasi-Poisson $A$-module). A quasi-Poisson $A$-coring is an $A$-coring $\mathcal C$ which is a $U(A)$-module such that the $A$-biaction turns $\mathcal C$ into a quasi-Poisson $A$-module and the coproduct $\Delta_{\mathcal C}$ and the counit $\epsilon_{\mathcal C}$ are right $U(A)$-linear. We deduce that a quasi-Poisson $A$-coring is nothing else a $(U(A),A)$-coring. Let $\mathcal C$ be a quasi-Poisson $A$-coring. A vector space $M$ is a right quasi-Poisson $\mathcal C$-comodule if $M$ is a right $\mathcal C$-comodule and a right $U(A)$-module such that with the right $A$-module structure, $M$ is a right quasi-Poisson $A$-module and the comodule structure map is $U(A)$-linear. We deduce that a right quasi-Poisson $\mathcal C$-comodule is nothing else than a right $(U(A),{\mathcal C})$-comodule. We conclude that the category $_{U(A)}{\mathcal M}^{\mathcal C}$ of right $(U(A),{\mathcal C})$-comodules over a $(U(A),A)$-coring $\mathcal C$ is isomorphic to the category ${\mathcal M}^{{\mathcal P},{\mathcal C}}$ of right quasi-Poisson comodules over the quasi-Poisson $A$-coring $\mathcal C$, since the two categories have the same morphisms which are the $U(A)$-linear and right $\mathcal C$-colinear maps. Let $\mathcal C$ be flat as left $A$-module. Then Theorems 3.2 and 3.3 give necessary and sufficient conditions for an object of ${\mathcal M}^{{\mathcal P},{\mathcal C}}$ to be projective (flat) over the endomorphism ring $End^{ {\mathcal P},{\mathcal C}}(\Lambda)$, where $\Lambda$ is a right quasi-Poisson $\mathcal C$-comodule that is a finitely generated (finitely presented) object in ${\mathcal M}^{{\mathcal P},{\mathcal C}}$. Note that if $\Lambda$ is finitely generated (finitely presented) as a left ${(^{\prime}{\mathcal C})}^{op} \#U(A)$-module, then it is a finitely generated (finitely presented) object in ${\mathcal M}^{{\mathcal P},{\mathcal C}}$. Thus we get a generalization of the results in \cite{Guedenon}. 

\subsection{$H$ is the group algebra of a group}
Let $G$ be a group acting by automorphismss on an associative $k$-algebra $A$. Then $A$ is a left $G$-module algebra, or equivalently, a left $kG$-module algebra, where $kG$ is the group algebra of $G$. We can define the notions of a $(kG,A)$-coring $\mathcal C$ which we can call a $(G,A)$-coring and of a right $(kG,{\mathcal C})$-comodule which we can call a right $(G,{\mathcal C})$-comodule. Thus we can consider the category $_{G}{\mathcal M}^{\mathcal C}$ of right $(G,{\mathcal C})$-comodules. Let $\mathcal C$ be flat as left $A$-module. Then Theorems 3.2 and 3.3 give necessary and sufficient conditions for projectivity and flatness over the endomorphism ring $_{G}End^{\mathcal C}(\Lambda)$, where $\Lambda$ is a right $(G,{\mathcal C})$-comodule that is a finitely generated (finitely presented) object in $_{kG}{\mathcal M}^{\mathcal C}$ (which is the case if $\Lambda$ is finitely generated (finitely presented) left $(^{\prime}{\mathcal C})^{op} \# kG$-module).

\subsection{$H$ is the enveloping algebra of a Lie algebra}
Let $\mathcal G$ be Lie algebra over a field $k$ acting by derivations on an associative algebra $A$. Then $A$ is a left $\mathcal G$-module algebra, or equivalently, a left $U({\mathcal G})$-module algebra, where $U({\mathcal G})$ is the universal enveloping algebra of $\mathcal G$. We can define the notions of a $(U({\mathcal G}),A)$-coring $\mathcal C$ (a $({\mathcal G},A)$-coring) and of a right $(U({\mathcal G}),{\mathcal C})$-comodule (a right $({\mathcal G},{\mathcal C})$-comodule). Thus we can consider the category $_{\mathcal G}{\mathcal M}^{\mathcal C}$ of right $({\mathcal G},{\mathcal C})$-comodules. Let $\mathcal C$ be flat as left $A$-module. Then Theorems 3.2 and 3.3 give necessary and sufficient conditions for projectivity and flatness over the endomorphism ring $_{\mathcal G}End^{\mathcal C}(\Lambda)$, where $\Lambda$ is a right $({\mathcal G},{\mathcal C})$-comodule that is a finitely generated (finitely presented) object in $_{U({\mathcal G})}{\mathcal M}^{\mathcal C}$ (which is the case if $\Lambda$ is finitely generated (finitely presented) left $(^{\prime}{\mathcal C})^{op} \# U({\mathcal G})$-module). The Lie algebra ${\mathcal G}$ could be the Lie algebra of vector fields on a smooth manifold $M$, and $A$ the commutative algebra $C^{\infty}(M)$ of smooth functions on $M$, where the action is the Lie derivative on functions. We recall that ${\mathcal G}$ and the one-forms $\Omega^1$ are left $(U({\mathcal G}),C^{\infty}(M))$-bimodules \cite[Example 2.4]{Schenkel 2} and \cite{AschSch, Schenkel 1}. For an $H$-module algebra $A$, left $(H,A)$-bimodules and their deformations play a central role in the study of the differential geometry of noncommutative manifolds \cite{Schenkel 2, AschSch, Schenkel 1}. 

\subsection{The category $_H{\mathcal M}_A^{H}$ of two-sided ($A,H$)-Hopf modules}

Let $k$ be a commutative ring and $H$ be a Hopf $k$-algebra with a bijective antipode. An algebra $A$ is a right {\it $H$-comodule algebra} if $A$ is a right $H$-comodule ($\rho_A(a)=a_0 \otimes a_1$) satisfying  
$$(aa')_0 \otimes (aa')_1=(a_0a'_0) \otimes (a_1a'_1) \quad \hbox{and} \quad (1_A)_0 \otimes (1_A)_1=1_A \otimes 1_H \quad \forall a, a' \in A.$$ 
Let $A$ be a right $H$-comodule algebra. A $k$-module $M$ is a right $(A,H)$-Hopf module if $M$ is a right $A$-module and a right $H$-comodule ($\rho_M(m)=m_0 \otimes m_1$) such that 
$$(ma)_0 \otimes (ma)_1=(m_0a_0) \otimes (m_1a_1) \quad \forall a \in A, m \in M.$$ 
It is easy to see that $A$ is a right $(A,H)$-Hopf module whenever $A$ is a right $H$-comodule algebra. For further information on $H$-coactions, we refer to \cite{Montgomery, caeMilZhu, brWi}. A $k$-module $M$ is a left-right $H$-Hopf module if $M$ is a left $H$-module and a right $H$-comodule such that 
$$(hm)_0 \otimes (hm)_1=(h_1m_0) \otimes( h_2m_1) \quad \forall h \in H, m \in M.$$ 
An $(H,A)$-bimodule is a $k$-module $M$ which is a left $H$-module and a right $A$-module such that $$(hm).a=h(m.a) \quad \forall h \in H, a \in A, m \in M.$$
A two-sided $(A,H)$-Hopf module is an $(H,A)$-bimodule which is a right $(A,H)$-Hopf module and a left-right $H$-Hopf module. Following \cite{Menini}, we will denote by $_H{\mathcal M}_A^{H}$ the category of two-sided $(A,H)$-Hopf modules with morphisms the left $H$-linear, right $A$-linear and right $H$-colinear maps.

Let $A$ be a trivial left $H$-module algebra. In the following proposition, $A \otimes_kH \otimes_k H$ is equipped with its structure of an $H$-module given by the diagonal $H$-action, that is, 
$$h''(a \otimes h \otimes h')=a \otimes (h''_1h) \otimes (h''_2h').$$

\begin{proposition} Let $A$ be a trivial left $H$-module algebra and a right $H$-comodule algebra. Set ${\mathcal C}=A \otimes_kH$. Then 

(1) ${\mathcal C}$ is an $(H,A)$-coring with grouplike element $1_A \otimes 1_H$: the $H$-action is defined by $h'.(a \otimes h)=a \otimes (h'h)$, the $A$-bimodule action is given by $a' \blacktriangleright(a \otimes h) \blacktriangleleft a''=(a'aa''_0) \otimes (ha''_1)$, the comultiplication is $\Delta_{\mathcal C}=id_A \otimes \Delta_H$ and the counit is $\epsilon_{\mathcal C}=id_A \otimes \epsilon_H$;

(2) the category $_H{\mathcal M}^{\mathcal C}$ is isomorphic to the category $_H{\mathcal M}_A^{H}$.

(3)  the category $_H{\mathcal M}^{H \otimes H}$ is isomorphic to the category $_H{\mathcal M}_H^{H}$ of two-sided Hopf modules, where we consider $H$ as an $H$-algebra with a trivial $H$-action.

\end{proposition}

\begin{proof} (1) By \cite[33.2]{brWi}, $\mathcal C$ is an $A$-coring with the given $A$-bimodule structure. We have
$$\begin{array}{rcl} h'.((a \otimes h) \blacktriangleleft a')&=& h'(aa'_0 \otimes ha'_1)\\
&=&(aa'_0 \otimes h'ha'_1)=(a\otimes h'h)a'\\
&=&(a\otimes h'_1h) \epsilon_H(h'_2)a'\\
&=&(h'_1.(a \otimes h)) \blacktriangleleft (h'_2.a'),\end{array}$$ and $A \otimes H$ is a left-right $(H,A)$-module. Clearly, $A \otimes H$ is a left $(H,A)$-module. It follows that $A \otimes H$ is a left $(H,A)$-bimodule. It is also clear that $\Delta_{\mathcal C}$ and $\epsilon_{\mathcal C}$ are $H$-linear.

(2) the category ${\mathcal M}^{\mathcal C}$ of right $\mathcal C$-comodules is isomorphic to the category ${\mathcal M}_A^{H}$ of right $(A,H)$-Hopf modules \cite[33.2]{brWi}. We deduce that the category $_H{\mathcal M}^{\mathcal C}$ is isomorphic to the category $_H{\mathcal M}_A^{H}$, since they have the same objects and the same morphisms. The result follows.

(3) Take $A=H$.
\end{proof}

Let $k$ be a field. Then $A \otimes H$ is left $A$-free. Theorems 3.2 and 3.3 give necessary and sufficient conditions for projectivity and flatness over the endomorphism ring $_HEnd_A^{H}(\Lambda)={}_HEnd^{A \otimes H}(\Lambda)$ of a right two-sided $(A,H)$-Hopf module $\Lambda$ if $\Lambda$ is finitely generated (finitely presented) in $_H{\mathcal M}_A^{H}$. Note that $_HEnd_A^{H}(A)$ is the subring $A^{HcoH}$ of $H$-invariants and $H$-coinvariants of $A$.

\subsection{The category $_{{\mathcal G}}{\mathcal M}_A^H$ of $(A,{\mathcal G},H)$-comodules}
In this subsection, $k$ is a field. We refer to \cite{Gordienko} and \cite{GUED} for more details on $(A,{\mathcal G},H)$-comodules.

\begin{definition} Let $H$ be a Hopf algebra, and $\mathcal G$ a Lie algebra. A vector space $A$ is a $({\mathcal G},H)$-comodule algebra if $A$ a ${\mathcal G}$-algebra (the ${\mathcal G}$-action is denoted by $\diamond$) and an $H$-comodule algebra such that 
$$(X \diamond a)_0 \otimes (X \diamond a)_1=(X \diamond a_0) \otimes a_1; \forall X \in {\mathcal G}, a \in A \eqno(3).$$
\end{definition}

\begin{definition} Let $H$ be a Hopf algebra, ${\mathcal G}$ a Lie algebra and $A$ a $({\mathcal G},H)$-comodule algebra. A vector space $M$ is an $(A,{\mathcal G},H)$-comodule if $M$ is a right $A$-module, a ${\mathcal G}$-module and an $H$-comodule such that the following three conditions are satisfied:
$$(ma)_0 \otimes (ma)_1=(m_0a_0) \otimes (m_1a_1); \forall a \in A, m \in M \eqno(4);$$
$$X \diamond(ma)=(X \diamond m)a+m(X \diamond a); \forall a \in A, m \in M, X \in {\mathcal G} \eqno(5)$$
and
$$(X \diamond m)_0 \otimes (X \diamond m)_1=(X \diamond m_0) \otimes m_1; \forall m \in M, X \in {\mathcal G} \eqno(6).$$
\end{definition}

Let us denote by $_{{\mathcal G}}{\mathcal M}_A^H$ the category of $(A,{\mathcal G},H)$-comodules with morphisms the right $A$-linear, $\mathcal G$-linear, $H$-colinear maps. Since $A$ is a ${\mathcal G}$-algebra, we can define the notion of a $({\mathcal G},A)$-coring.

\begin{proposition}  Let $H$ be a  Hopf algebra, $\mathcal G$ a Lie algebra and $A$ a $({\mathcal G},H)$-comodule algebra.

(1) ${\mathcal C}=A \otimes H$ is a $({\mathcal G},A)$-coring: the ${\mathcal G}$-module and the $A$-coring structures are given by
$$X \diamond(a \otimes h)=(X \diamond a) \otimes h; \forall X \in {\mathcal G}, a \in A, h \in H \quad \hbox{and}$$
$$a' \blacktriangleright (a \otimes h) \blacktriangleleft a''=(a'aa''_0) \otimes (ha''_1); \forall a,a',a'' \in A, h \in H.$$
 
(2) The category $_{\mathcal G}{\mathcal M}^{\mathcal C}$ of right $({\mathcal G}, {\mathcal C})$-comodules and the category 
$_{{\mathcal G}}{\mathcal M}_A^H$ of $(A,{\mathcal G},H)$-comodules are isomorphic. 

\end{proposition}

\begin{proof} (1) Clearly, ${\mathcal C}$ is a left $A \#U({\mathcal G})$-module and a right $A$-module for the given $\mathcal G$-action and the $A$-biaction. For all $X \in {\mathcal G}$, $a,a' \in A$ and $h \in H$, we have
$$\begin{array}{rcl} X \diamond ((a \otimes h) \blacktriangleleft a')&=& X \diamond((aa'_0) \otimes (ha'_1))\\
&=&[(X \diamond a)a'_0+a(X \diamond a'_0)] \otimes (ha'_1)\\
&=&[(X \diamond a)a'_0] \otimes(ha'_1)+[a(X \diamond a'_0)] \otimes (ha'_1)\\
&=&[(X \diamond a) \otimes h] \blacktriangleleft a'+a(X \diamond a')_0 \otimes h(X \diamond a')_1\\
&=&[(X \diamond a) \otimes h] \blacktriangleleft a'+[(a \otimes h) \blacktriangleleft (X \diamond a')],\end{array}$$
this shows that $\mathcal C$ is a left-right $({\mathcal G},A)$-module. It follows that $\mathcal C$ is a left $({\mathcal G},A)$-bimodule. It is well known that $A \otimes H$ is an $A$-coring for the given $A$-biaction with comultiplication
$$\Delta_{\mathcal C}(a \otimes h)=(a \otimes h_1) \otimes_A(1_A \otimes h_2)=a \otimes h_1 \otimes h_2; \forall a \in A, h \in H,$$
and counit
$$\epsilon_{\mathcal C}(a \otimes h)=\epsilon_H(h)a; \forall a \in A, h \in H.$$
It is easy to see that $\Delta_{\mathcal C}$ and $\epsilon_{\mathcal C}$ are $\mathcal G$-linear maps.

(2)  If $M$ is an object of $_{\mathcal G}{\mathcal M}^{\mathcal C}$, then $M$ is a $\mathcal G$-module, a right $(A \otimes H)$-comodule in such a way that it is a left-right $({\mathcal G},A)$-module, and the comodule structure map $\rho_{M,{\mathcal C}}: M \rightarrow M \otimes_A(A \otimes H); m \mapsto m_{(0)} \otimes_A m_{(1)}$ is $\mathcal G$-linear and $A$-linear. Since $M$ is a right $\mathcal C$-comodule, it is well known that it is a right $(A,H)$-Hopf module. The $H$-comodule structure $\rho_{M,H}$ is the composition of $\rho_{M,{\mathcal C}}$ followed by the canonical isomorphism $M \otimes_A(A \otimes H) \simeq M \otimes H$. The fact that $M$ is a left-right $({\mathcal G},A)$-module means that the relation (5) is satisfied. The fact that $\rho_{M,{\mathcal C}}$ is $\mathcal G$-linear and right $A$-linear means that the relations (4) and (6) are satisfied. It follows that $M$ is an object of $_{{\mathcal G}}{\mathcal M}_A^H$. Conversely, if $M$ is an object of $_{{\mathcal G}}{\mathcal M}_A^H$, then $M$ is a $\mathcal G$-module, a right $(A,H)$-Hopf module, and it satisfies the relations (4), (5) and (6), that is, $M$ is a $\mathcal G$-module, a right $\mathcal C$-comodule, and the relations (4), (5) and (6) are satisfied. The right $\mathcal C$-comodule structure is defined by $m  \mapsto m_{(0)} \otimes_A m_{(1)}=m_0 \otimes_A(1_A \otimes m_1)$. So $M$ is a $\mathcal G$-module, a right $\mathcal C$-comodule, a left-right $({\mathcal G},A)$-module and $\rho_{M,{\mathcal C}}$ is a homomorphism of $\mathcal G$-modules. Therefore $M$ is an object of $_{\mathcal G}{\mathcal M}^{\mathcal C}$.

The morphisms of $_{\mathcal G}{\mathcal M}^{\mathcal C}$ are the left $\mathcal G$-linear, right $\mathcal C$-comodule maps, that is, the $\mathcal G$-linear maps which are also homomorphisms of $(A,H)$-Hopf modules. It follows that the categories 
$_{\mathcal G}{\mathcal M}^{\mathcal C}$ and $_{{\mathcal G}}{\mathcal M}_A^H$ have the same morphisms.
\end{proof}

Note that $A \otimes H$ is left $A$-free since $k$ is a field. Then Theorems 3.2 and 3.3 give necessary and sufficient conditions for projectivity and flatness over the endomorphism ring $_{{\mathcal G}}End_A^{H}(\Lambda)={}_{\mathcal G }End^{A \otimes H}(\Lambda)$ of an $(A,{\mathcal G},H)$-comodule $\Lambda$, where $\Lambda$ is finitely generated (finitely presented) in $_{{\mathcal G}}{\mathcal M}_A^H$. Note that $_{{\mathcal G}}End_A^{H}(A)=A^{{\mathcal G},H}$ the subring of $\mathcal G$-invariants and $H$-coinvariants of $A$.

\subsection{Deformation of an $(H,A)$-coring}

Let $k$ be a field and $H$ a Hopf algebra. A twist ${\mathcal F}$ is an element ${\mathcal F} \in H \otimes H$ that is invertible and that satisfies two compatibility conditions \cite[Definition 12.1]{Schenkel 1} or \cite[Definition 2.5]{Schenkel 2}. By \cite[Theorem 12.2]{Schenkel 1} or \cite[Theorems 2.6 and 2.8]{Schenkel 2}, the twist of a Hopf algebra $H$ leads to a new Hopf algebra $H^{\mathcal F}$. All left $H$-modules $M$ are deformed to left $H^{\mathcal F}$-modules $M_{\star}$. All left $H$-module algebras $A$ are deformed to left $H^{\mathcal F}$-module algebras $A_{\star}$ \cite[Theorem 12.3]{Schenkel 1}, all left $(H,A)$-modules, left-right $(H,A)$-modules and left $(H,A)$-bimodules $M$ are deformed to left $(H^{\mathcal F},A_{\star})$-modules, left-right $(H^{\mathcal F},A_ {\star})$-modules and left $(H^{\mathcal F},A_{\star})$-bimodules $M_{\star}$ \cite[Theorem 12.4]{Schenkel 1}. Let $A$ be an $H$-module algebra and $\mathcal C$ an $(H,A)$-coring. Then $\mathcal C$ and ${\mathcal C} \otimes_A{\mathcal C}$ are left $(H,A)$-bimodules. Thus we get the left $(H^{\mathcal F},A_{\star})$-bimodules ${\mathcal C}_{\star}$, $({\mathcal C} \otimes_A{\mathcal C})_{\star}$ and ${\mathcal C}_{\star} \otimes_{A_{\star}} {\mathcal C}_{\star}$. The coproduct $\Delta_{\mathcal C}$ of $\mathcal C$ is a left $(H,A)$-bimodule homomorphism from $\mathcal C$ to ${\mathcal C} \otimes_A{\mathcal C}$. Its deformation $D_{\mathcal F}(\Delta_{\mathcal C})$ is a left $(H^{\mathcal F},A_{\star})$-bimodule homomorphism from ${\mathcal C}_{\star}$ to $({\mathcal C} \otimes_A{\mathcal C})_{\star}$ \cite[Theorems 3.5 and 4.7]{AschSch} or \cite[Theorem 13.2]{Schenkel 1}. By \cite[Lemma 5.19]{AschSch}, the left $(H^{\mathcal F},A_{\star})$-bimodules $({\mathcal C} \otimes_A{\mathcal C})_{\star}$ and ${\mathcal C}_{\star} \otimes_{A_{\star}} {\mathcal C}_{\star}$ are isomorphic. Therefore we get a left $(H^{\mathcal F},A_{\star})$-bimodule homomorphism from ${\mathcal C}_{\star}$ to 
 ${\mathcal C}_{\star} \otimes_{A_{\star}} {\mathcal C}_{\star}$. The counit $\epsilon_{\mathcal C}$ is a left $(H,A)$-bimodule homomorphism from $\mathcal C$ to $A$. Its deformation $D_{\mathcal F}(\epsilon_{\mathcal C})$ is a left $(H^{\mathcal F},A_{\star})$-bimodule homomorphism from ${\mathcal C}_{\star}$ to $A_{\star}$. It follows that ${\mathcal C}_{\star}$ is an $(H^{\mathcal F},A_{\star})$-coring which we will consider as the deformation of $\mathcal C$. We can then define right $(H^{\mathcal F},{\mathcal C}_{\star})$-comodules. Let $H$ be a Hopf algebra with a bijective antipode. By Lemma 2.9, the left dual $^{\prime}{\mathcal C}$ is a left $H$-module algebra and a left $(H,A)$-bimodule. So its deformation $(^{\prime}{\mathcal C})_{\star}$ is a left $H^{\mathcal F}$-module algebra and a left $(H^{\mathcal F},A_{\star})$-bimodule. By \cite[Example 4.9]{AschSch} or \cite[Example 3.3]{Schenkel 2}, $(^{\prime}{\mathcal C})_{\star}$ and $^{\prime}({{\mathcal C}_{\star}})$ are isomorphic as left $(H^{\mathcal F},A_{\star})$-bimodules. They are also isomorphic as left $H^{\mathcal F}$-algebras. Then Theorems 3.2 and 3.3 give necessary and sufficient conditions for projectivity and flatness over the endomorphism ring $_{H^{\mathcal F}}End^{{\mathcal C}_{\star}}(\Lambda)$, where $\Lambda$ is a right $(H^{\mathcal F},{\mathcal C}_{\star})$-comodule finitely generated (finitely presented) in the category of right $(H^{\mathcal F},{\mathcal C}_{\star})$-comodules (which is the case if $\Lambda$ is finitely generated (finitely presented) as a left $(^{\prime}({\mathcal C}_{\star}))^{op} \#H^{\mathcal F}$-module).

\bigskip
 
Let us end the paper by the following remark.

\begin{remark} We refer to \cite{Bohm}, \cite{BrMil}, \cite{Rang} and \cite{Schauen} for more information on bialgebroids. We keep the notations and the conventions above: $H$ is a bialgebra and $A$ is an $H$-algebra. By \cite{Bohm}, $G=A \otimes H \otimes A$ is a left $A$-bialgebroid (Connes-Moscovici's bialgebroid) with an associative multiplication 
$$(a \otimes h \otimes u)(a' \otimes h' \otimes u')=a_1(h_1.a') \otimes h_2h' \otimes (h_3.u')u$$
with unit $1_A \otimes 1_H \otimes 1_A$, with source and target maps
$$A \rightarrow G, \quad a \mapsto a \otimes 1_H \otimes 1_A \quad  \hbox{and} \quad A^{op} \rightarrow G, \quad a \mapsto 1_A \otimes 1_H \otimes a,$$
respectively. Coproduct and counit are given by (see also Lemma 2.4, where $H$ is just a coalgebra)
$$a \otimes h \otimes u \mapsto (a \otimes h_1 \otimes 1_A) \otimes_A(1_A \otimes h_2 \otimes u) \quad \hbox{and} \quad a \otimes h \otimes u \mapsto a\epsilon_H(h)u.$$
A homomorphism of left $G$-modules is an $H$-linear $A$-bimodule homomorphism. As in \cite{Schauen}, we can assert that the category $_{H,A}{\mathcal M}_A$ of left $(H,A)$-bimodules and the category $_G{\mathcal M}$ of left $G$-modules are isomorphic: $M$ is a left $G$-module if and only if $M$ is a left $(H,A)$-bimodule: $$a(hm)u=(a \otimes h \otimes u)m, \quad \forall a, u \in A, \quad h \in H.$$
If ${\mathcal C}$ is an $(H,A)$-coring, we can show that $\Delta_{\mathcal C}$ and $\epsilon_{\mathcal C}$ are $G$-linear. Thus an $(H,A)$-coring $\mathcal C$ is exactly a coalgebra in the monoidal category $(_G{\mathcal M}, \otimes_A,A$, i.e $\mathcal C$ is a $G$-module coalgebra. Set $${\mathcal D}= {\mathcal C} \otimes_AG \cong {\mathcal C} \otimes H \otimes A.$$ 
Then it is well known that $\mathcal D$ is a $G$-coring. The comodules over $\mathcal D$ are left $(H,A)$-bimodules with a compatible $\mathcal D$-coaction: as we can see, they are very closely related to the right $(H, {\mathcal C})$-comodules considered in this paper. Maybe this observation can help to deduce the results presented in this paper from the results in \cite{CG} and \cite{TG}.  
 \end{remark}

\end{document}